\documentclass[12pt]{amsart}

\usepackage[marginratio=1:1,height=600pt,width=404pt,tmargin=110pt]{geometry}

\usepackage{amsmath, amsfonts, amssymb, amscd, bezier, amstext, amsthm}
\usepackage[latin1]{inputenc}
\usepackage{color}
\usepackage{graphicx}
\usepackage{fancyhdr}
\usepackage{array}
\usepackage{esint}
\usepackage{hyperref}
\usepackage{enumitem}

\newcommand{\R}{\mathbb{R}}

\newcounter{alpha}

\theoremstyle{plain}

\newtheorem{theorem}{Theorem}[section]

\newtheorem{lemma}[theorem]{Lemma}
\newtheorem{corollary}[theorem]{Corollary}
\newtheorem{proposition}[theorem]{Proposition}

\newtheorem{remark}[theorem]{Remark}

\title[Nonlinear elliptic systems with critical exponent]{Qualitative properties of positive singular solutions to nonlinear elliptic systems with critical exponent}

\author[R. Caju]{Rayssa Caju}

\author[J.M.\ do \'O]{Jo\~ao Marcos do \'O}

\author[A. Santos]{Almir Silva Santos}

\address[R. Caju]{Department of Mathematics,
	\newline\indent 
	Federal University of Para\'{\i}ba
	\newline\indent
	58051-900, Jo\~ao Pessoa-PB, Brazil}
\email{\href{mailto:rayssacaju@gmail.com}{rayssacaju@gmail.com}}

\address[J.M. do \'O]{Department of Mathematics,
	\newline\indent 
	Federal University of Para\'{\i}ba
	\newline\indent
	58051-900, Jo\~ao Pessoa-PB, Brazil}
\email{\href{mailto:jmbo@pq.cnpq.br}{jmbo@pq.cnpq.br}}

\address[A. S. Santos]{Department of Mathematics,
	\newline\indent 
	Federal University of Sergipe
	\newline\indent
	49100-000, S\~ao Cristov\~{a}o-SE, Brazil}
\email{\href{mailto:almir@mat.ufs.br}{almir@mat.ufs.br}}
\date{}

\numberwithin{equation}{section}

\begin{document}
	
%
%

\begin{abstract}
	We studied the asymptotic behavior of local solutions for strongly coupled critical elliptic systems near an isolated singularity.
	For the dimension less than or equal to five we prove that any singular solution is asymptotic to a rotationally symmetric Fowler type solution. 
	This result generalizes the celebrated work due to Caffarelli, Gidas, and Spruck \cite{CGS} who studied asymptotic proprieties to the classic Yamabe equation. 	In addition, we generalize similar results by Marques \cite{marques} for inhomogeneous context, that is, when the metric is not necessarily conformally flat.
\end{abstract}

\maketitle

%
%


%
%

\section{Introduction}

Let $g$ be a smooth  Riemannian metric on the unit ball $B^{n}_{1}(0)\subset \mathbb R^{n} (n\geq 3 )$. 
In this paper we derive asymptotic behavior for positive singular solutions to the
following class of nonlinear elliptic system with critical exponent defined inhomogeneous context
\begin{equation}\label{S}
-\Delta_{g}u_{i} + \sum_{j=1}^{2}A_{ij}(x)u_{j} = \frac{n(n-2)}{4}|\mathcal{U}|^{\frac{4}{n-2}}u_{i}, \;\;\; \mbox{for } i=1,2,
\end{equation}
in the punctured ball $\Omega = B^{n}_{1}(0)\backslash\{0\}$. Here $ \mathcal{U}=(u_1,u_2)$,  $\Delta_g= \mathrm{div}_g \nabla $ is the Laplace Beltrami operator of the metric $g$ and $A$ is a $C^{1}$ map from $\Omega$  to $M_{2}^{s}(\mathbb R)$ the vector space of symmetrical $2\times 2$ real matrices. We can write $A = (A_{ij})_{i,j}$, where each $A_{ij}$ is a $C^{1}$ real valued function.
We say that $ \mathcal{U}=(u_1,u_2)$ is a \emph{positive} solution of \eqref{S} if each coordinate $u_{i}$ is a positive function for $i=1,2$. 
Here $ |\mathcal{U}| $ denotes de Euclidean norm of $\mathcal{U}$, that is, $ |\mathcal{U}|^2=u_1^2+u_2^2$.

O. Druet and E. Hebey \cite{DH} have investigated and proved analytic stability for strongly coupled critical elliptic systems in the inhomogeneous context of a compact Riemannian manifold. This kind of coupled system of nonlinear Schrödinger equations are part of important branches of mathematical physics like in fiber-optic theory and in the behavior of deep water waves and freak waves in the ocean. See \cite{DH} and \cite{DHV} and the reference therein. The critical systems $\eqref{S}$ are weakly coupled by the linear matrix $A$ and strongly coupled by the Gross--Pitaevskii type nonlinearity in the right-hand side of \eqref{S}. Moreover, system \eqref{S} can be seen as a natural generalization of the singular Yamabe problem for an arbitrary metric.

We also mention the recent work of Z. Chen and C.-S. Lin \cite{CL}, where they have studied a different two-coupled elliptic system with critical exponent, although also is an extension of Yamabe-type  equation in the strongly coupled regime.

We call a solution $ \mathcal{U} $ of \eqref{S} singular at origin (nonremovable singularity) if 
\begin{equation}\label{ls1}
\lim_{|x|\rightarrow 0} |\mathcal{U}(x)| = +\infty.
\end{equation}

 For easy reference, we introduce now the following hypotheses on the potential 
\begin{enumerate}[label=({H\arabic*})]
	\item \label{H1} The coefficient $A_{ij}$ are nonpositive for any $i\not=j$, which is equivalent to say that $-A$ is \emph{cooperative}.
	\item  \label{H2} In dimension $n=5$, there exists a $C^2$-function $f$ such that
	\begin{equation*}
	A(x) = f(x)Id_2 + O(|x|)
	\end{equation*}
	near the origin, where $Id_2$ is the identity matrix.
\end{enumerate}

To understand our main result in this work we consider the Fowler solutions (or Delaunay-type solutions) which are radial solutions of
\begin{equation}\label{eq021}
\Delta u+\frac{n(n-2)}{4}u^{\frac{n+2}{n-2}}=0 \quad \mbox{in} \quad \mathbb R^n\backslash\{0\},
\end{equation}
which blow-up at zero. Here $\Delta$ denote the Euclidean Laplacian. By  Caffarelli, Gidas and Spruck in \cite{CGS}, it is well known that any solution to \eqref{eq021} with a nonremovable singularity is rotationally invariant.

In order to derive asymptotic properties for singular positive solutions of \eqref{S} it is crucial  in our argument to study qualitative properties of the limit system
\begin{equation}\label{LS}
\Delta u_i + \frac{n(n-2)}{4}|\mathcal{U}|^{\frac{4}{n-2}}u_i=0\quad  \mbox{in} \quad \mathbb R^n \backslash \{0\} , \; i=1,2,
\end{equation}
which is obtained after blowing up the system \eqref{S}. First of all we prove that any solution of \eqref{LS} is radially symmetric.

\begin{theorem}[Radial Symmetry]\label{RS1} Let $n \geq 3$ and $\mathcal{U} \in C^2(\mathbb R^n \backslash \{0\})$ be a nonnegative singular solution of \eqref{LS}. Then $\mathcal{U}$ is radially symmetric about the origin.
\end{theorem}

Using this theorem we are capable to prove the following classification result 
\begin{theorem}[Classification]\label{class1} Let $n \geq 3$ and $ \mathcal{U} \in C^2(\mathbb R^n \backslash \{0\})$ be a nonnegative singular solution of \eqref{LS}.
	Then there exists  $\Lambda \in \mathbb{S}^{1}_{+} = \{x \in \mathbb{S}^{1} : x_{i} \geq 0\}$ and a Fowler solution $u_0$ such that 
	\begin{equation*}
	\mathcal{U} = u_{0}\Lambda.
	\end{equation*}
\end{theorem}

We will call a solution of the limit system \eqref{LS} as a {\it Fowler-type solution}. 

Finally, our main result reads as follows
\begin{theorem}[Main Result]\label{main} Assume $3 \leq n \leq 5$ and that the potential A satisfies the hypotheses \ref{H1} and \ref{H2}. Let $\mathcal{U} \in C^2(\Omega)$ be a positive singular solution of $\eqref{S}$.
Then there exists $\alpha > 0$ and a Fowler-type solution $\mathcal U_0$ such that
\begin{equation*}
\mathcal{U}(x) = (1 + O(|x|^{\alpha}))\mathcal{U}_{0}(x)\quad \mbox{as} \quad x \rightarrow 0.
\end{equation*}
\end{theorem}

As a consequence of Theorem~\ref{main} we have that any positive solution $\mathcal{U}$  for the system 
\begin{equation}\label{LSB}
\Delta u_i + \frac{n(n-2)}{4}|\mathcal{U}|^{\frac{4}{n-2}}u_i=0\quad  \quad \mbox{in} \quad \Omega,  \; i=1,2,
\end{equation}
either $\mathcal{U}$ can be smoothly extended to the origin or there is a Fowler-type solution $\mathcal{U}_{0}$ and there is $\alpha>0$ such that
\begin{equation*}
\mathcal{U}(x) = (1 + O(|x|^{\alpha}))\mathcal{U}_{0}(x)\quad \mbox{as} \quad x \rightarrow 0.
\end{equation*}


In \cite{CGS}, L. Caffarelli, B. Gidas and J. Spruck have proved that any positive solution to
$$\Delta u+\frac{n(n-2)}{4}u^{\frac{n+2}{n-2}}=0\;\;\mbox{ in }\Omega$$
either it can be smoothly extended to the origin or it is asymptotically to some Fowler solution $u_0$ when $x$ goes to zero. Later N. Korevaar, R. Mazzeo, F Pacard and R. Schoen in \cite{KMPS} have improved this result by proving that 
$$u(x)=(1+O(|x|^\alpha))u_0(x)\;\;\mbox{ as }x\rightarrow 0,$$
for some $\alpha>0$. F. C. Marques in \cite{marques} extended this result to a more general setting, at least in low dimensions. More precisely, he has proved that Fowler solutions still serve as asymptotic model to the equation
$$\Delta_gu-\frac{n-2}{4(n-1)}R_gu+\frac{n(n-2)}{4}u^{\frac{n+2}{n-2}}=0\;\;\mbox{ in }\Omega,$$
for $3\leq n\leq 5$. C. C. Chen and C. S. Lin in a series of works, \cite{MR1466584, MR1642113, MR1679784, MR1737506}, have studied local singular solutions to the prescribed scalar curvature equation
$$\Delta u+K(x)u^{\frac{n+2}{n-2}}=0,$$
where $K$ is a positive $C^1$ function defined on a neighborhood around zero. They have studied conditions on the function $K$ such that the Fowler solution still serve as asymptotic models. In the spirit of the previous works, the main motivation of this paper is the question that whether these asymptotic result could be extended to the system \eqref{S}. 

The first task to answer this question is to find an asymptotic model to the system. In \cite{DHV}, O. Druet, E. Hebey and J. Vetóis have proved that any nonnegative entire solutions of the system 
\begin{equation}\label{dh}
\Delta u_{i} + |\mathcal{U}|^{\frac{4}{n-2}}u_{i} = 0 \quad \mbox{in} \quad  \mathbb R^{n},
\end{equation}
is of the form $U=u\Lambda$, where $\Lambda \in\mathbb S^1_+=\{(x,y)\in\mathbb R^2;x^2+y^2=1\mbox{ with }x,y\geq 0\}$ and $u$ is a entire solution of \eqref{eq021}. This result motivated us to prove the classification result to the limit system, Theorem \ref{class1}, which is the asymptotic model to the system \eqref{S}.

Now let us briefly describe the strategy of the proof of the Theorem \ref{main}, which is inspired on the works \cite{KMPS} and \cite{marques}.  The first step of our proof is to prove the Theorems \ref{RS1} and \ref{class1}. In Section \ref{jf}, based on this theorems we can do a carefully analysis of the Jacobi fields to limit System \eqref{LS} similar to the analysis done in \cite{KMPS}. A key result to the proof of the Theorem \ref{main} is the following fundamental upper bound.
\begin{theorem}\label{14}
Suppose $3\leq n\leq 5$. Assume that $\mathcal{U} = (u_{1},u_{2})$ is a positive solution of $\eqref{S}$ in $\Omega = B^{n}_{1}(0)\backslash \{0\}$. There exists a constant $c>0$ such that 
	\begin{equation}\label{u1}
	|\mathcal{U}(x)| \leq cd_{g}(x,0)^{\frac{2-n}{2}},
	\end{equation}
	for $0 < d_{g}(x,0) <1/2$.
\end{theorem}
We use the Method of Moving Planes to prove this upper bound. Also, like in \cite{marques} the difficulty here is the fact that our system \eqref{S} has no symmetries. Inspired by the work of F. C. Marques \cite{marques} we can overcome this difficulty at least for lower dimensions. As a consequence, if we assume that the potential satisfies \ref{H1}, we prove a uniform spherical Harnack inequality around the singularity, see Corollary \ref{DesHarnack}.

As a consequence of the upper bound \eqref{u1} we are able to define a Pohozahev-type invariant as
\begin{equation*}
P(\mathcal{U}) = \lim_{r\rightarrow \infty} P(r,\mathcal{U}),
\end{equation*}
where
\begin{equation}\label{P1}
P(r,\mathcal{U})= \displaystyle \int_{\partial B_{r}}\left(\frac{n-2}{2}\left\langle \mathcal U,\frac{\partial \mathcal U}{\partial \nu}\right\rangle- \frac{r}{2}|\nabla \mathcal U|^{2}+ r\left|\frac{\partial \mathcal U}{\partial\nu}\right|^{2} \right.\left.+\displaystyle r\frac{(n-2)^2}{8}|\mathcal{U}|^{\frac{2n}{n-2}}\right)d\sigma.\end{equation}

Inspired by the works of L. Caffarelli, B. Gidas and J. Spruck in \cite{CGS}, we prove a removable singularity theorem, proving that this invariant is always nonpositive and it is equal to zero if and only if the singularity is removable. More precisely, we prove the following result

\begin{theorem}\label{t5} Assume $3 \leq n \leq 5$ and that the potential $A$ satisfies the hypotheses \ref{H1} and \ref{H2}. Let $\mathcal{U}$ be a positive solution to the system \eqref{S} in $B^{n}_{1}(0)\backslash\{0\}$. Then $P(\mathcal{U}) \leq 0$. Moreover, $P(\mathcal{U}) = 0$ if and only if the origin is a removable singularity.
\end{theorem}

  An important consequence of this result is the
  \begin{corollary}\label{cc1}
   Assume $3 \leq n \leq 5$ and that the potential $A$ satisfies the hypotheses \ref{H1} and \ref{H2}. Let $\mathcal{U}$ be a positive solution to the system \eqref{S} in $B^{n}_{1}(0)\backslash\{0\}$. If the origin is a nonremovable singularity, then there exists $c > 0$ such that
\begin{equation}\label{l1}
	|\mathcal{U}(x)| \geq cd_{g}(x,0)^{\frac{2-n}{2}}
\end{equation}
	for each $i$ and $0 < d_{g}(x,0) < 1/2$.	
\end{corollary}

\begin{remark}
We can see that Theorems \ref{RS1}, \ref{class1} and  \ref{14} hold for systems with any number of equations by using easy modifications of our proofs.
\end{remark}

The paper is organized as follows. In Section \ref{clre} we use the Method of Moving Planes to prove that singular solutions of \eqref{LS} are radially symmetric. Therefore System \eqref{LS} is equivalent to a ODE system which is crucial in our argument to prove the classification result, Theorem \ref{class1}. We also study growth properties of the Jacobi fields in order to obtain the asymptotic behavior of singular solutions of \eqref{S}. In Section \ref{pb} we prove the upper bound, Theorem \ref{14}, Theorem \ref{t5} and the lower bound, Corollary \ref{cc1} to solutions of \eqref{S}. In Section \ref{crs} we conclude the prove of Theorem \ref{main}. Precisely, by using a delicate argument, originally due to Leon Simon and applied in \cite{KMPS} and \cite{marques}, we apply the estimates \eqref{u1} and \eqref{l1} to prove that any solution of our System \eqref{S} is asymptotically to some Fowler-type solution.

\noindent \textbf{Acknowledgments:}
\section{Classification result for the limit system}\label{clre}
Our main goal in this section is to study some properties of singular nonnegative solutions $\mathcal{U} = (u_{1},u_{2})$ to the following system
\begin{equation}\label{eq025}
\Delta u_{i} + \frac{n(n-2)}{4}| \mathcal{U} |^{\frac{4}{n-2} }u_{i}=0 \quad \mbox{in} \quad \mathbb R^{n}\setminus \{0\}.
\end{equation}
Similarly to the definition for the system $\eqref{S}$ we say that a solution $\mathcal U$ of \eqref{eq025} is singular if
$
\lim_{|x|\rightarrow 0} |\mathcal{U}(x)| = +\infty.
$
It may happen that only one of the coordinates blows up at the singularity.
The above system will be important in our analysis since its solutions will play a similar role to Fowler solutions in the singular Yamabe problem. 

One of our main result in this section it is the characterization of nonnegative singular solutions of the limit System \eqref{eq025}. We will show that any solution of this system is a multiple of a vector in the unit sphere by a classical Fowler solution, see Theorem \ref{class} below.
\subsection{Radial Symmetry}

The first step to reach our goal is to show that any solution of the system \eqref{eq025} is radially symmetric. To this end we will use the Method of Moving Planes.

\begin{theorem}[Radial Symmetry]\label{RS} Let $n \geq 3$ and $\mathcal{U} = (u_{1},u_{2})$ be a nonnegative $C^2$ singular solution of the nonlinear elliptic system \eqref{eq025}. Then $\mathcal{U}$ is radially symmetric about the origin.
\end{theorem}


\begin{proof} The proof will follow the argument in \cite{CGS} and \cite{CL}. Without loss of generality we may assume that $\displaystyle\lim_{|x|\rightarrow 0}u_1(x) = +\infty$. 
	Fix an arbitrary $z \neq 0$ in $\mathbb R^{n}$. Define the Kelvin transform 
	\begin{equation*}
	U_{i}(x) = \frac{1}{|x|^{n-2}}u_{i}\left( z + \frac{x}{|x|^{2}}\right), \;\;\mbox{ for }i=1,2.
	\end{equation*}
	 We see that $U_1$ is singular at zero and $z_0=-z/|z|^2$ and $U_2$ is singular at zero. Moreover, the Kelvin transform satisfies 
	\begin{equation}\label{LLS}
	\begin{aligned}
	\Delta U_i   + & \; \frac{n(n-2)}{4}|\mathbf{U}|^{\frac{4}{n-2}}U_i=0   \quad  & \mbox{in} & \quad \mathbb{R}^n\setminus\{0,z_{0}\}.
	\end{aligned}
	\end{equation}
	It is easy to see that
	$$U_i(x)=O(|x|^{2-n})\;\;\mbox{ as }\;|x|\rightarrow\infty$$
	and that each $U_i$ has the following harmonic asymptotic expansion at infinity
	$$\begin{array}{lll}
	U_i(x)  & =  & a_{i0}|x|^{2-n}+a_{ij}x_j|x|^{-n}+O(|x|^{-n}),\\
	\partial_{x_j}U_i  &=  & (2-n)a_{i0}x_j|x|^{-n}+O(|x|^{-n}),\\
	\partial_{x_k x_j}^2U_i &= & O(|x|^{-n}),
	\end{array}$$
	where $a_{i0}=u_i(z)$ and $a_{ij}=\partial_{y_j}u_i(z)$.

	We will show that $U_{i}$ are axisymmetric with respect to the axis going through 0 and $z$. Choose a reflection direction orthogonal to this axis and assume without loss of generality that it is equal to the positive $x_{n}$ direction $(0,\ldots,0,1)$. For $\lambda > 0$ consider $\Sigma_{\lambda} :=\{ x \in \mathbb R^{n}: x_{n} > \lambda\}$ and the reflection 
	\begin{equation*}
	x = (x_{1},\ldots,x_{n-1},x_{n}) \in \Sigma_{\lambda}\mapsto x_{\lambda} = (x_{1},\ldots,x_{n-1}, 2\lambda - x_{n}).
	\end{equation*}
	
	Using the asymptotic expansion at infiniy for $U_i$ and Lemma 2.3 in \cite{CGS}, we see that there exist positive constants $\overline{\lambda}$ and $R$ such that for any $\lambda \geq \overline{\lambda}$ we have
	\begin{equation}\label{p11}
	U_{i}(x) < U_{i}(x_{\lambda}), \mbox{ for } x \in \Sigma_{\lambda} \mbox{ and } |x_\lambda| > R.
	\end{equation}
	By Lemma 2.1 in \cite{ChenLin} there exists a constant $C>0$ such that 
	\begin{equation}\label{p2}
	U_{i}(x)\geq C, \ \ \ \mbox{for} \ \ \ x \in \overline{B_{R}} \setminus \{0,z_{0}\}. 
	\end{equation}
	
	Since $U_{i}(x) \rightarrow 0$  as $|x| \rightarrow + \infty$, then by \eqref{p11} and \eqref{p2} we have that there exists $\lambda_{0} > \overline{\lambda}$ such that when $|x|\geq 2\lambda_{0} - R$ it holds $U_{i}(x) < C $. On the other hand if $x\in \Sigma_{\lambda}$ and $x_{\lambda} \in \overline{B}_{R}$ then $x \notin B_{2\lambda - R}$. Thus for every $\lambda \geq \lambda_{0}$ we have 
	\begin{equation}\label{tree}
	U_{i}(x) \leq  U_{i}(x_{\lambda}), \mbox{ for all }  x \in \Sigma_{\lambda} \mbox{ and }  x_{\lambda} \not\in \{0,z_{0}\}.
	\end{equation}  
	Let 
	\begin{equation*}
	\lambda^{*} :=\inf\{ \tilde{\lambda} > 0;  \eqref{tree} \mbox{ holds for all } \lambda \geq \tilde{\lambda} \}.
	\end{equation*}
	It suffices to show that $\lambda^{*} = 0$. Indeed, this proves that $U_{i}$ is axisymmetric with respect to the axis going through $0$ and $z$ and since $z$ is arbitrary, each $u_{i}$ is radially symmetric about the origin.
	
	Suppose by contradiction that $\lambda^{*} > 0$. Then \eqref{tree} holds for $\lambda = \lambda^{*}$. Since $U_{1}$ goes to infinity when $x$ goes to $z_{0}$ we see that $U_{1}$ cannot be invariant by the reflection $x_{\lambda^{*}}$. Then by applying the maximum principle to $U_{i}(x_{\lambda^{*}}) - U_{i}(x)$ in $\eqref{LLS}$, we conclude that 
	\begin{equation}\label{four}
	U_{i}(x) < U_{i}(x_{\lambda^*}), \mbox{ for }  x \in \Sigma_{\lambda^*} \mbox{ and }  x_{\lambda^*} \not\in \{0,z_{0}\}.
	\end{equation}
	Note that $0,z_0\not\in\partial\Sigma_{\lambda^*}$, since $\lambda^*>0$.Then applying the Hopf boundary lemma for $x \in \partial\Sigma_{\lambda^{*}}$ we obtain
	\begin{equation}\label{five}
	\partial_{x_{n}}(U_{i}(x_{\lambda^{*}}) - U_{i}(x)) = -2\partial_{x_{n}} U_{i}(x) > 0.
	\end{equation}
	
	Now choose sequences $\lambda_{j}\nearrow \lambda^{*}$ and  $x^{j} \in \Sigma_{\lambda_{j}}$ such that $U_{1}(x^{j}_{\lambda_{j}}) < U_{1}(x^{j})$. By Lemma 2.4 in \cite{CGS}, we conclude that the sequence $|x^{j}|$ is bounded. Hence passing to a subsequence we may assume that $x_{j} \rightarrow \overline{x} \in \overline{\Sigma_{\lambda^{*}}}$ with $U_{1}(\overline{x}_{\lambda^{*}}) \leq U_{1}(\overline{x})$. By \eqref{four} we know that $\overline{x} \in \partial\Sigma_{\lambda^{*}}$ and then $\partial_{x_{n}} U_{1}(\overline{x}) \geq 0$, a contradiction with \eqref{five}. Therefore, $\lambda^{*} = 0$.
\end{proof}


Since we proved the radial symmetry of the solutions we will use a change of variables to approach our problem using ODE analysis. This is the subject from the next subsection. 
\subsection{ODE Analysis}

If $\mathcal{U} = (u_{1}, u_{2})$ is a solution of the system $\eqref{RS}$, then by Theorem \ref{RS} we have that $u_{i}(x) $ depends only on $|x|$. Let 
\begin{equation*}
v_{i}(t) := e^{-\delta t}u_{i}(e^{-t}\theta),\;\;i=1,2,
\end{equation*}
where $t = -\ln r$, $\theta=x/|x|$ and $\delta = (n-2)/2$. Then, it is easy to see that $V= (v_{1},v_{2})$ satisfies the ordinary differential equation
\begin{equation}\label{ODELS}
v_{i}'' - \delta^{2}v_{i} + \frac{n(n-2)}{4}|V|^{\frac{4}{n-2}}v_{i} = 0,\;\;i=1,2,
\end{equation}
with $t\in\mathbb R$, $v_i\geq 0$ and $v_i\in C^2(\mathbb R)$. Setting $w_i:=v_i'$ this ODE system transformed into a first order Hamiltonian system
$$\left\{\begin{array}{rcl}
v_i' & = & w_i\\
w'_i & = & \displaystyle\delta^2v_i-\frac{n(n-2)}{4}|V|^{\frac{4}{n-2}}v_i,
\end{array},\;\;i=1,2.
\right.$$
Setting $W=(w_1,w_2)$, we can see that the Hamiltonian energy given by
$$H(V,W)=\frac{1}{2}\left(|W|^2-\delta^2|V|^2+\delta^2|V|^{\frac{2n}{n-2}}\right),$$
is constant along solutions of \eqref{ODELS}. It does not seem possible to integrate completely the ODE \eqref{ODELS}, as in the classical case. Besides the Hamiltonian energy above, it does not seem possible to find other quantities which are preserved along the trajectories of \eqref{ODELS}. This is the reason why we will rely on the following approach to prove the Theorem \ref{class1}.

Since the Hamiltonian energy above is constant along any solution $V$ of the ODE \eqref{ODELS}, then $  \Psi(t):=H(V(t),V'(t))=K$, for some constant $K$.
In fact, by a direct computation, we easily deduce that
\[
  \Psi(t)=\frac{P(r,\mathcal U)}{\sigma_{n-1}},
\]
where $\sigma_{n-1}$ is the volume of the unit sphere $\mathbb S^{n-1}$ and $P(r,\mathcal U)$ is the Pohozaev integral defined in \eqref{P1}.
As a consequence it is easy to see that $V$ and $|V'|$  are bounded.  Indeed, if we suppose that there exists a sequence $t_{n}\rightarrow + \infty$ such that $|V(t_{n})| \rightarrow + \infty$, we have 
\begin{equation*}
\frac{1}{2}\left( \left(\frac{|V'(t_{n})|}{|V(t_{n})|}\right)^{2}-\delta^2\right) + \frac{\delta^2}{2}|V(t_{n})|^{\frac{4}{n-2}} = \frac{K}{|V(t_{n})|^{2}}
\end{equation*}
which is a contradiction. Similarly, we conclude that $|V'|$ is bounded.

The first consequence of the radial symmetry of solutions is the fact that each coordinate $u_i$ satisfies the following dichotomy: either it is strictly positive or it vanishes identically.


\begin{lemma}\label{L1} Let $\mathcal{U}$ be a nonnegative singular solution of the limit system \eqref{eq025}. If there exists $y \in \mathbb R^{n}\backslash \{0\}$ such that $u_{i}(y) = 0$, then $u_{i}\equiv 0$.
\end{lemma}
\begin{proof}
We already know that any solution $\mathcal U$ of the limit system \eqref{eq025} is radially symmetric. Suppose that there exist $y$ and $z$ in $\mathbb R^{n}\setminus \{0\}$, with $|z| > |y|$ and  $u_{i}(z)>u_{i}(y) = 0$.  By the maximum principle from Chen and Lin in \cite[Lemma 2.1]{ChenLin}, we know that 
\begin{equation*}
u_{i}(x) \geq \inf_{\partial B_{|z|}}u_{i} = u_{i}(z) > 0,
\end{equation*}
for all $x \in B_{|z|}\backslash\{0\}$, which contradicts the fact that $u_{i}(y) = 0$. Hence, if there exists a point $y$ such that $u_{i}(y) = 0$ then $u_{i}$ must be zero for all $z$ with $|z| \geq |y|$.
However, we know that the corresponding function $v_{i}$ is a solution for the ODE system $\eqref{ODELS}$, and by the uniqueness of solutions to the ODE system, it would be identically zero.
\end{proof}

In order to prove the classification result, we will need the following technical results.
\begin{lemma}\label{lemma3.1} If $v_{i}$ has a limit $C_{i}$ at $+\infty$ or $- \infty$, then $C_{i} \leq \left(\frac{n-2}{n}\right)^{\frac{n-2}{4}}$. Moreover, if $\Psi(t)=K\geq 0$, then $C_i=0$.
\end{lemma}
\begin{proof} Assume that $\displaystyle \lim_{t\rightarrow +\infty}v_{1}(t) = C_{1} > 0$. For each coordinate consider a sequence of translates $v_{i,k}(t) = v_{i}(t+k)$. By \eqref{ODELS}, up to passing to a subsequence, we may assume that $v_{1,k}\rightarrow C_{1}$ and $v_{2,k}^2 \rightarrow w \geq 0$ uniformly in compacts with respect to the $C^{2}$ topology. Thus 
$$-\delta^{2} C_{1} + \frac{n(n-2)}{4}(C_{1}^{2} + w)^{\frac{2}{n-2}}C_{1} = 0,$$
 which implies that $C_{1} \leq \left(\frac{n-2}{n}\right)^{\frac{n-2}{4}}$ and $w$ is also constant, $C_2$ for example. This implies that
 $$K=-\frac{\delta^2}{2}|C|^2+\frac{\delta^2}{2}|C|^{\frac{2n}{n-2}}\geq 0,$$
 where $C=(C_1,C_2)$. Therefore, $C_i=0$.
\end{proof}

Now, for each $i$ define the auxiliary functions $f_{i} : \mathbb{R} \rightarrow \mathbb{R}$ by
\begin{equation}\label{f}
f_{i}(t) := - \frac{1}{2}(v_{i}'(t))^{2} + \frac{\delta^{2}}{2}v_{i}(t)^{2} - \frac{\delta^2}{2}v_{i}(t)^{\frac{2n}{n-2}}.
\end{equation} 
By $\eqref{ODELS}$ it is easy to see that  
$$f_{i}'(t) = \frac{\delta^2}{2}\left(|V|^{\frac{4}{n-2}}(t) - v_{i}(t)^{\frac{4}{n-2}}\right)v_{i}(t)v_{i}'(t).$$ 
Since $|V|\geq v_i$, then the monotonicity of $f_{i}$ is exactly the same as the monotonicity of $v_{i}$. This monotonicity property is very important and it will be used in the sequel.

\begin{lemma}\label{lemma3.2} For all $t$ in $\mathbb R$ we have that
	\begin{equation*}
	v_{i}(t) < 1.
	\end{equation*}
\end{lemma}
\begin{proof}
Suppose by contradiction that the result is not true. By Lemma \ref{lemma3.1} there exists a local maximum point $t_{0} \in \mathbb{R}$ of $v_{1}$ such that $v_{1}(t_{0}) \geq 1$. This implies that 
\begin{equation*}
f_{1}(t_{0}) = \frac{\delta^{2}}{2}v_{1}^{2}(t_{0}) - \frac{\delta^2}{2}v_{1}(t_{0})^{\frac{2n}{n-2}} \leq 0.
\end{equation*}
We claim that there exists $t_{1} > t_{0}$ such that $v_{1}'(t) < 0$ for $t \in (t_{0},t_{1})$ and $v_{1}'(t_{1}) = 0$. Otherwise it would follow from Lemma \ref{lemma3.1} that $v_{1} \downarrow C_{1}\leq ((n-2)/n)^{\frac{n-2}{4}}$ and $f_{1}(t) \downarrow -C < 0$ as $t$ goes to $+ \infty$. Then we have that $v_{1}'(t) \rightarrow 0$ as $t$ goes to $+\infty$ and so 
\begin{equation*}
0 > -C = \frac{\delta^{2}}{2}C_{1}^{2} - \frac{\delta^2}{2}C_{1}^{\frac{2n}{n-2}} > 0,
\end{equation*}
which is a contradiction. Hence there exists such $t_{1} $ and it satisfies $v_{1}''(t_{1})  \geq 0$. On the other hand
\begin{equation*}
f_{1}(t_{1}) = \frac{\delta^{2}}{2}v_{1}(t_{1})^{2} - \frac{\delta^2}{2}v_{1}(t_{1})^{\frac{2n}{n-2}} < f_{1}(t_{0}) \leq 0,
\end{equation*}
namely $v_{1}(t_{1}) > 1$ which implies from $\eqref{ODELS}$ that $v_{1}''(t_{1}) < 0$, which is again a contradiction. Therefore, such $t_0$ can not exist.
\end{proof}

\subsection{Classification Result}

In the beggining of this section, we mentioned that the solutions of our limit system play a similar role to the Fowler solutions in the singular Yamabe problem. We will summarize briefly basic properties about Fowler solutions. 

Remember that a positive function $u$ is a Fowler or Delaunay-type solution if it satisfies
\begin{equation}\label{sss}
\Delta u + \frac{n(n-2)}{4}u^{\frac{n+2}{n-2}}=0 \quad \mbox{in}\quad \mathbb R^{n}\setminus \{0\}.
\end{equation}

It is well known by  \cite[Theorem 8.1]{CGS}  that $u$ is rotationally invariant, and thus the equation can be reduced to an ODE. Since $\mathbb R^{n}\backslash \{0\}$ is conformally diffeomorphic to a cylinder, relative to the cylindrical metric this equation becomes
\begin{equation}\label{fow}
	v'' - \frac{(n-2)^2}{4}v + \frac{n(n-2)}{4}v^{\frac{n+2}{n-2}} = 0.
\end{equation}
The analysis of this equation can be done by converting into a system of first order equations
\begin{equation*}
\left\{
\begin{array}{rcl}
v'& =&w\\
w' &=& \displaystyle\frac{(n-2)^2}{4}v - \frac{n(n-2)}{4}v^{\frac{n+2}{n-2}}.
\end{array}
\right.
\end{equation*}
whose Hamiltonian energy, given by
\begin{equation*}
	H(v,w) = w^2 - \frac{(n-2)^2}{4}v^2 + \frac{(n-2)^2}{4}v^{\frac{2n}{n-2}}
\end{equation*}
is constant along solutions of $\eqref{fow}$. 

By examining the level curves of the energy, we see that all positive solutions lies in the bounded set $\{H < 0\}\cap\{v>0\}$. The basic properties of this solutions are summarized in the next proposition whose proof can be found in \cite{MP}.
\begin{proposition} For any $H_{0} \in (-((n-2)/n)^{n/2}(n-2)/2,0)$, there exists a unique bounded solution of $\eqref{fow}$ satisfying $H(v,v') = H_0$, $v'(0) = 0$ and $v''(0) > 0$. This solution is periodic and for all $t\in \R$ we have $v(t) \in (0, 1)$. This solution can be indexed by the parameter $\varepsilon = v_\varepsilon(0) \in(0, ((n-2)/n)^{(n-2)/4})$, which is the smaller of the two values $v_\varepsilon$ assumes when $v'_\varepsilon(0) = 0$. 
\end{proposition}

In \cite{DHV}, Druet, Hebey and Vetóis have proved a characterization result for entire solutions in $\mathbb R^{n}$ for the limit system \eqref{eq025}. They proved that any solution in $\mathbb R^n$ of the equation \eqref{eq025} is of the form $u\Lambda$, where $u$  is a entire solution of the equation \eqref{fow} and $\Lambda=(x,y)$ with $x^2+y^2=1$ and $x,y>0$.  Inspired by their result, we may ask whether a similar description holds in the singular case. Indeed, we will prove that every nonnegative singular solution of \eqref{eq025} is a Fowler solution, up to a vector in the unit sphere with nonnegative coordinates.

\begin{theorem}\label{class}
	Suppose that $\mathcal{U}$ is a nonnegative singular solution for \eqref{eq025}. Then there exists $\Lambda \in \mathbb{S}^{1}_{+}$ and a Fowler solution $u$ such that 
	\begin{equation*}
	\mathcal{U}= u\Lambda.
	\end{equation*}
	where $\mathbb{S}^{1}_{+} = \{x \in \mathbb{S}^{1} : x_{i} \geq 0\}$
\end{theorem}
\begin{proof}
	We will prove the result in the case where $v_{1}$ and $v_{2}$ are positive solutions since otherwise the result follows directly from Lemma \ref{L1}.
	
	Let ${V} = (v_{1}, v_{2})$ be a solution for \eqref{ODELS} obtained from $\mathcal{U}$ after the change of variables of Fowler. We know that it satisfies
	\begin{equation*}
	v_{i}'' - \delta^{2}v_{i} + \frac{n(n-2)}{4}|{V}|^{\frac{4}{n-2}}v_{i} = 0, \ \ t \in \mathbb{R}.
	\end{equation*}
	From this, we can note that 
	\begin{equation*}
	v_{1}''v_{2} = v_{1}v_{2}''
	\end{equation*}
	which implies
	\begin{equation*}
	v_{1}'v_{2} - v_{1}v_{2}' = c = constant.
	\end{equation*}
	Suppose that $c \neq 0$. Without loss of generality we can assume that $c > 0$. Since $V$ is bounded, then by Lemma \ref{lemma3.2} we obtain
	\begin{equation*}
	\left(\frac{v_1}{v_2}\right)' = \frac{c}{v_{2}^2} \geq c
	\end{equation*}
	which is a contradiction, since the quotient will be a monotone function and thus it would assume negative values. Therefore we conclude that 
	\begin{equation*}
	\left(\frac{v_{1}}{v_{2}}\right)' = 0
	\end{equation*}
	and so $v_{1}/v_{2} = \eta$ is constant. It follows that
	\begin{equation*}
	{V} = (v_{1},v_{2}) = (\eta v_{2},v_{2}) = (\eta, 1)v_{2} = \sqrt{1+\eta^2}v_{2}\left(\frac{\eta}{\sqrt{1+\eta^2}}, \frac{1}{\sqrt{1+\eta^2}}\right)
	\end{equation*}
	and a direct computation shows that $v_{0} = \sqrt{1+\eta^2}v_{2}$ is a Fowler solution.
	
\end{proof}

This result will be extremely important in the study of the Jacobi fields. The Theorem \ref{class} will be very useful to simplify the analysis of the linearized system. Also it will allow us to use all the properties of Fowler solutions that we described previously.

Let us mention some important consequences of the Theorem \ref{class}. The first result is concerns the singularity of the solutions.


\begin{corollary}
Let $\mathcal{U} = (u_{1},u_{2})$ be a positive singular solution to the equation \eqref{eq025}. Then $\displaystyle\lim_{|x|\rightarrow 0}u_{i}(x) = +\infty$, for $i =1,2$.
\end{corollary}

The second direct consequences is the following

\begin{corollary}
If $\mathcal{U}$ is a nonnegative singular solution to the system \eqref{eq025}, then  there exist positive constants $c_{1}$ and $c_{2}$ such that  
\begin{equation}\label{c1}
c_{1}|x|^{\frac{2-n}{2}} \leq |\mathcal{U}(x)| \leq c_{2}|x|^{\frac{2-n}{2}}.
\end{equation}
\end{corollary}

To understand the third and last consequence, given a solution $\mathcal U$ of \eqref{eq025} consider, 
It is an easy computation to see that 
\begin{equation}\label{pohid}
P(r,\mathcal{U}) = P(s, \mathcal{U}),
\end{equation}
for any $r$ and $s$, where $P(r,\mathcal U)$ is defined in \eqref{P1}. Hence we can denote it by $P(\mathcal{U})$ which is called the Pohozaev invariant of the solution $\mathcal U$. Besides note that by the classification result, there exist a Fowler solution $u_{0}$ and an unit vector $\Lambda = (\Lambda_{1}, \Lambda_{2})$ with nonnegative coordinates, such that $\mathcal{U} = u_{0}\Lambda$. Consequently $P(\mathcal{U})$ coincides with the Pohozaev invariant of $u_{0}$ as defined in \cite{CGS} and we obtain the following important result

\begin{corollary}\label{classical} Let $\mathcal{U}=(u_{1} , u_{2})$ be a nonnegative solution of $\eqref{sss}$. Then $P(\mathcal{U}) \leq 0$. Moreover, $P(\mathcal{U}) = 0$ if and only if, $u_{i} \in C^{2}(\mathbb R^{n})$, for all $i$.
\end{corollary}

\subsection{Jacobi Fields for the Limit System}\label{jf}

In the study of the asymptotic behavior of the solution to the singular Yamabe problem, N. Korevaar et al. \cite{KMPS} and F. Marques \cite{marques} have used the growth properties of Jacobi fields as a tool to obtain properties of the singular solutions. In order to prove our main theorem, we need to do a similar analysis. In this case we are going to linearize the equation about a solution given by Theorem \ref{class}. In order to do this, let us first recall some details about the analysis in the simple case.

Let us briefly describe the analysis done in \cite{KMPS}. Consider the equation \eqref{sss} 
which using the cylindrical coordinates change is equivalent to 
\begin{equation}\label{jf2}
\partial_{t}^2 v + \Delta_{\mathbb{S}^{n-1}}v - \frac{(n-2)^2}{4}v + \frac{n(n-2)}{4}v^{\frac{n+2}{n-2}} = 0,
\end{equation}
where $u(x)=|x|^{\frac{2-n}{2}}v(\frac{x}{|x|},-\log|x|)$. The authors in \cite{KMPS} considered the linearization of the operator in $\eqref{jf2}$ around a Fowler solution $v_{\varepsilon}$ which is given by
	\begin{equation}\label{jf3}
	L_{\varepsilon} = \partial_{t}^{2} + \Delta_{\mathbb{S}^{n-1}} - \frac{(n-2)^2}{4} + \frac{n(n+2)}{4}v_{\varepsilon}^{\frac{4}{n-2}}.
	\end{equation}
Since the above operator has periodic coefficients, it could be studied using the classical Flocquet theory or also by separation of variables, see \cite{RS}. The elements in the kernel of the linearized operator $L_\varepsilon$, that is, the functions $\psi$ satisfying $L_{\varepsilon}\psi = 0$, are called \textbf{Jacobi Fields}. 
The linearization \eqref{jf3} was studied using results due to R. Mazzeo, D. Pollack and K. Uhlenbeck in \cite{MPU} based on the spectral decomposition of the laplacian operator $\Delta_{\mathbb{S}^{n-1}}$.
	
If $\{\lambda_{k}, \mathcal{X}_{k}(\theta)\}$ are the eigendata of $\Delta_{\mathbb{S}^{n-1}}$, using the convention that this eigenvalues are listed with multiplicity, we can write $\lambda_{0} = 0$, $\lambda_{1} = \cdots = \lambda_{n} = n-1$, $\lambda_{n+1} = 2n$ and so on. Hence the linearized operator could be decomposed into many ordinary differential operators given by
	\begin{equation}\label{proj}
	L_{\varepsilon,k} = \partial_{t}^{2} + \left(\frac{n(n+2)}{4}v_{\varepsilon}^{\frac{4}{n-2}} - \frac{(n-2)^2}{4} - \lambda_{k}\right).
	\end{equation}
		
With the decomposition of the laplacian in terms of the eingenvalues, it is sufficient consider solutions to the induced problems $L_{\varepsilon,k}(\psi_k) = 0$, where $\psi(t,\theta)=\sum_k\psi_k(t)\chi_k(\theta)$.
To the first eingenvalue $\lambda_{0} = 0$, we consider the families of solutions to (\ref{jf2}) given by
\begin{equation}\label{fam}
	T \rightarrow v_{\varepsilon}(t +T) \quad\mbox{ and } \quad\varepsilon\rightarrow v_{\varepsilon}(t),
\end{equation}
then differentiate these families with respect to the parameters we obtain solutions to (\ref{proj}) for $k=0$ given by 
\begin{equation*}
\psi_{\varepsilon,0}^{+}(t) = \frac{d}{dT}v_{\varepsilon}(t+T) = v'_{\varepsilon}(t+T), \quad \psi_{\varepsilon,0}^{-}(t) = \frac{d}{d\varepsilon}v_{\varepsilon}(t)
\end{equation*}
which are linearly independents Jacobi fields with periodic and linearly growth, respectively. 

Using a similar construction for $\lambda_{k}=n-1$ , they were able to build linearly independent solutions $\psi_{\varepsilon,k}^{\pm}$ that are exponentially increasing and decreasing. Finally when $k > n+1$, we known that the term of zero order of the operator in \eqref{proj} is negative, since $\lambda_{k}\geq 2n$ and $v_{\varepsilon} < 1$, which implies that $L_{\varepsilon,k}$ satisfies the maximum principle. 

Once we have listed the aforementioned properties, we can begin the study of the system \eqref{eq025}. Consider the operator in cylindrical coordinates $H({V}) = (H^{1}({V}), H^{2}({V}))$ where
\begin{equation}\label{jf5}
H^{i}({V}) = \partial_{t}^{2}v_{i} + \Delta_{\mathbb{S}^{n-1}}v_{i} - \frac{(n-2)^2}{4}v_{i} + \frac{n(n-2)}{4}|\mathcal{V}|^{\frac{4}{n-2}}v_{i}.
\end{equation}
The linearization of this operator on a solution ${V}_{0} = (v_{0,1}, v_{0,2})$ is given by $\mathcal{L}(\phi) = (\mathcal{L}^{1}(\phi), \mathcal{L}^{2}(\phi))$ where
\begin{equation*}
\begin{array}{rcl}
\mathcal{L}^{i}(\phi) &=  &\displaystyle\frac{\partial}{\partial t}\bigg|_{t=0} H^{i}({V}_{0} + t\phi)  =  \displaystyle\partial_{t}^{2}\phi_{i} + \Delta_{\mathbb{S}^{n-1}}\phi_{i} - \frac{(n-2)^2}{4}\phi_{i} \\
& & \displaystyle
+ \frac{n(n-2)}{4}\left(\frac{4}{n-2}|{V}_{0}|^{\frac{4}{n-2}-2}\langle{V}_{0}, \phi\rangle v_{0,i} + |{V}_{0}|^{\frac{4}{n-2}}\phi_{i}\right),
\end{array}
\end{equation*}
where $\phi=(\phi_1,\phi_2)$. However, by the Theorem \ref{class} there exists a Fowler solution $v_{\varepsilon}$ and $\Lambda = (\Lambda_{1},\Lambda_{2}) \in\mathbb S^1$ with $\Lambda_i>0$ such that $V_{0} = v_{\varepsilon}\Lambda$. Using this, the operator $\mathcal L^i$ can be simplified and written as
\begin{equation}\label{jf4}
\begin{aligned}
\mathcal{L}_{\varepsilon}^{i}(\phi) =   \partial_{t}^{2}\phi_{i} + \Delta_{\mathbb{S}^{n-1}}\phi_{i} - \frac{(n-2)^2}{4}\phi_{i} 
+ n\Lambda_{i}\langle \Lambda, \phi\rangle v_{\varepsilon}^{\frac{4}{n-2}} + \frac{n(n-2)}{4}v_{\varepsilon}^{\frac{4}{n-2}}\phi_{i}.
\end{aligned}
\end{equation}
Now we will study the \textbf{Jacobi fields} to the system \eqref{eq025}, that is, solutions of the linear ODE system $\mathcal L_{\varepsilon}({\phi}) = 0$.

Similarly to the case of a single equation, we consider the spectral decomposition of the operator $\Delta_{\mathbb{S}^{n-1}}$. Let $\phi = (\phi_{1},\phi_{2})$ be a solution of $\mathcal{L}_{\varepsilon}(\phi)=0$ and consider $\phi_{i}(\theta,t) = \sum \phi_{ik}(t)\mathcal{X}_{k}(\theta).$
Then $\phi_{ik}$ satisfies 
\begin{equation*}
	\mathcal{L}_{\varepsilon, k}\phi_{ik} + n\Lambda_{i}\langle \Lambda, ({\phi}_{1k},{\phi}_{2k})\rangle v_{\varepsilon}^{\frac{4}{n-2}} + \frac{n(n-2)}{4}v_{\varepsilon}^{\frac{4}{n-2}}\phi_{ik} = 0,
\end{equation*}
where $\mathcal{L}_{\varepsilon, k} = \partial_{t}^{2} - (\lambda_{k} + \frac{(n-2)^2}{4})$.


In the one equation case we saw that the two linearly independent fields were obtained as derivatives of one-parameter families of solutions of $\eqref{jf5}$. However, since we are dealing with an ODE linear system of second order with two equations, we expect to obtain four linearly independent Jacobi fields. 

Consider the following one-parameter families of solutions to the system \eqref{eq025} in cylindrical coordinates
\begin{equation*}
	T \rightarrow v_{\varepsilon}(t + T)\Lambda \quad\mbox{and}\quad \varepsilon \rightarrow v_{\varepsilon}\Lambda.
\end{equation*}

The derivatives of this families provide us the solutions $\phi_{\varepsilon, 0}^{1} = \psi_{\varepsilon, 0}^{+}\Lambda$ and $\phi_{\varepsilon, 0}^{2} = \psi_{\varepsilon, 0}^{-}\Lambda$, where $\psi_{\varepsilon, 0}^{\pm}$ are the previously described linearly independents Jacobi fields for one equation. Consequently $\phi_{\varepsilon,0}^1$ and $\phi_{\varepsilon,0}^2$ are also linearly independent and they are periodic and linearly growth, respectively.

Observe that there is another trivial one-parameter family of solutions to \eqref{eq025}. In fact, let $\Lambda(\alpha) = (\cos \alpha, \sin \alpha)$ with $\Lambda(0) = \Lambda$ be a path in the sphere and consider the family
\begin{equation*}
\alpha\rightarrow v_{\varepsilon}\,\Lambda(\alpha).
\end{equation*}
If we differentiate this family with respect to the parameter we get a third Jacobi field $\phi_{\varepsilon, 0}^{3} = v_{\varepsilon}\overline{\Lambda}$ which is a periodic function, where $\overline{\Lambda}\not=0$ is orthogonal to $\Lambda$. We automatically have that $\{\phi_{\varepsilon,0}^1, \phi_{\varepsilon,0}^2, \phi_{\varepsilon,0}^3\}$ is a linearly independent set.

Unfortunately, we are only able to construct the fourth Jacobi Field in an indirect way. First observe that $v_{\varepsilon}$ satisfies $\overline{L}_{\varepsilon,0}v_{\varepsilon} = 0$, where
\begin{equation} \label{eq:jf} 
	\overline{L}_{\varepsilon, 0} := \partial_{t}^{2} - \frac{(n-2)^2}{4}+ \frac{n(n-2)}{4}v_{\varepsilon}^{\frac{4}{n-2}}.
\end{equation}


Since this operator is linear we know that there exists a solution $\phi_{\varepsilon,0}^{-}$ to \eqref{eq:jf} which is linearly independent to $\phi_{\varepsilon,0}^{+} := v_{\varepsilon}$. Define $\phi_{\varepsilon,0}^{4} = \phi_{\varepsilon,0}^{-}\overline{\Lambda}$. It is easy to see that $\mathcal L_{\varepsilon,0}(\phi_{\varepsilon,0}^4)=0$. Observe that $\overline{L}_{\varepsilon, 0} < L_{\varepsilon, 0}$. Using standard methods of comparison of solutions to ODE, we obtain that the growth of $\phi_{\varepsilon,0}^-$ is at most linear, see \cite{KMPS}.

Therefore we constructed four linearly independent Jacobi fields to the system \eqref{jf4} with respect to the first eigenvalue.

Based on the above construction, for the higher eigenvalues $\lambda_{k}$ the Jacobi fields is given by  $\phi_{k}^{1} = \psi_{\varepsilon,k}^{+}\Lambda$, $\phi_{k}^{2} = \psi_{\varepsilon,k}^{-}\Lambda$, $\phi_{k}^{3} = \phi_{\varepsilon,k}^{+}\overline{\Lambda}$ and $ \phi_{k}^{4} = \phi_{\varepsilon,k}^{-}\overline{\Lambda}$ where $\psi_{\varepsilon,k}^{\pm}$ are the original solutions constructed for one equation which we already know the growth, see \cite{KMPS}, and $\phi_{\varepsilon,k}^{\pm}$ are solutions of 
\begin{equation*}
	\overline{L}_{\varepsilon,k} = \partial_{t}^{2}+ \frac{n(n-2)}{4}v_{\varepsilon}^{\frac{4}{n-2}} - \frac{(n-2)^2}{4} - \lambda_{k}
\end{equation*}
We know that the zero order term of the above operator is negative when $k>0$ since $v_{\varepsilon} <1$ and $\lambda_k\geq n-1$ for such $k$. This implies that $\overline L_{\varepsilon,k}$ satisfies the maximum principle, so we are able to determine the growth of these fields, concluding our analysis.


\section{Coupled elliptic system in the punctured ball}\label{pb}

Let $g$ be a smooth  Riemannian metric on the unit ball $B^{n}_{1}(0)\subset \mathbb R^{n}$ with $n\geq 3$. Consider a positive solution to the system 
\begin{equation}\label{eq000}
-\Delta_{g}u_{i} + \sum_{j=1}^{2}A_{ij}(x)u_{j} = \frac{n(n-2)}{4}|\mathcal{U}|^{\frac{4}{n-2}}u_{i}
\end{equation}
in the punctured ball $\Omega = B^{n}_{1}(0)\backslash\{0\}$.  Here $A$ is a $C^{1}$ map from the unit ball centered in the origin to the vector space of symmetrical $2\times 2$ real matrices and  $-A$ is cooperative, that is, the components in the nondiagonal $A_{ij}$ of $A$, $i\neq j$, are nonpositive.

\subsection{Upper and lower bounds near a singularity}

In this section we will obtain upper and lower bounds for solutions to our system defined in the punctured  ball. Since there exist a diffeomorphism between the half cylinder and the punctured ball it will be more convenient work in cylindrical coordinates. Explicitly, the diffeomorphism $\Phi : (\mathbb{R}\times S^{n-1}, g_{cyl} = dt^{2} + d\theta^{2}) \rightarrow (\mathbb R^{n}\backslash\{0\},\delta)$ is given by $\Phi(t,\theta)  = e^{-t}\theta$
with inverse $\Phi^{-1}(x)=(-\log|x|,x|x|^{-1})$. One also verifies that $\Phi^{*}\delta = e^{-2t}g_{cyl}$. 

Define $v_{i}(t,\theta) =  |x|^{\frac{n-2}{2}}u_{i}(x)$ and $\hat{g} = e^{2t}\Phi^{*}g = (e^{\frac{n-2}{2}t})^{\frac{4}{n-2}}\Phi^{*}g$. Using that
\begin{equation}\label{confc}
L_{v^{\frac{4}{n-2}}g}(u) = v^{-\frac{n+2}{n-2}}L_{g}(vu),
\end{equation}
where the linear operator $L_{g} = \Delta_{g} - \frac{n-2}{4(n-1)}R_{g}$ is the \emph{conformal Laplacian}, we obtain that the system \eqref{eq000} is equivalent to 
\begin{equation}\label{mud}
\mathcal{L}_{\hat{g}}(v_{i}) - \sum_{j=1}^{2}B_{ij}v_{j} + \frac{n(n-2)}{4}|\mathcal{V}|^{\frac{4}{n-2}}v_{i} = 0,
\end{equation}
where 
\begin{equation}\label{eq02}
\begin{array}{rcl}
\mathcal{L}_{\hat{g}}(v_{i}) &=&\displaystyle \Delta_{\hat{g}}v_{i} - \frac{n-2}{4(n-1)}(R_{\hat{g}} - e^{-2t}R_{\Phi^{*}g})v_{i},\\ 
B_{ij} &=& e^{-2t}A_{ij}\circ\Phi
\end{array}
\end{equation}
and $\mathcal{V} = (v_{1}, v_{2})$. 

It is also useful to remember that in cylindrical coordinates we have
\begin{equation}\label{cur}
R_{\hat{g}} - e^{-2t}R_{\Phi^{*}g} = (n-2)(n-1) + 2(n-1)e^{-t}\frac{\partial_{r}\sqrt{|g|}}{\sqrt{|g|}}\circ \Phi.
\end{equation}

The proof of the next result is strongly based on \cite{marques}.

\begin{theorem}\label{upper} Suppose $3\leq n\leq 5$. Assume that $\mathcal{U} = (u_{1},u_{2})$ is a positive solution of $\eqref{eq000}$ in $\Omega$. There exists a constant $c>0$ such that 
	\begin{equation}
	|\mathcal{U}(x)| \leq cd_{g}(x,0)^{\frac{2-n}{2}},
	\end{equation}
	for $0 < d_{g}(x,0) <1/2$.
\end{theorem}

\begin{proof} Given $x_{0} \in \Omega$ with $d_{g}(x_{0},0) <1/2$ and $s \in (0,1/4)$ such that $\overline{B_{s}(x_{0})} \subset \Omega$, define
	\begin{equation*}
	f(x) = (s - d_{g}(x,x_{0}))^{\frac{n-2}{2}}|\mathcal{U}(x)|,
	\end{equation*}
	for $x \in \overline{B_{s}(x_{0})}$. It suffices to show that there exists a positive constant $C$ such that any such $f$ satisfies $f(x) \leq C \mbox{ in }  B_{s}(x_{0})$. To see this, for $s=\frac{|x_{0}|}{2}$, we have $f(x_{0}) = s^{\frac{n-2}{2}}|\mathcal{U}(x_{0})| \leq C$. 
	
	The proof is by contradiction. So assume there is no such constant $C$. Then we can find a sequence of points $x_{0,k}$ and positive numbers $s_{k}$ such that, if $x_{1,k}$ denotes the maximum point of the corresponding $f_{k}$, we have
	\begin{equation*}
	f_{k}(x_{1,k}) = (s_{k} - d_{g}(x,x_{0,k}))^{\frac{n-2}{2}}|\mathcal{U}(x_{1,k})| \rightarrow \infty.
	\end{equation*}
	Note that, $0 < s_{k} < 1/4$ implies that $s_{k} - d_{g}(x,x_{0,k}) < 1/4$. This implies that $(s_{k} - d_{g}(x,x_{0,k}))^{\frac{n-2}{2}} < 2^{2-n}$ and therefore $2^{n-2}f_{k}(x) < |\mathcal{U}(x)|$. Then we conclude that $|\mathcal{U}(x_{1,k})| \rightarrow \infty$ and consequently $x_{1,k}\rightarrow 0$.
	
	Let $\varepsilon_{k} = |\mathcal{U}(x_{1,k})|^{-\frac{2}{n-2}}$ and define
	\begin{equation*}
	\tilde{u}_{i,k}(y) = \varepsilon_{k}^{\frac{n-2}{2}}u_{i}(\exp_{x_{1,k}}(\varepsilon_{k}y)).
	\end{equation*}
	for $i \in \{1, 2\}$. Note that $|\tilde{\mathcal{U}}_{k}(0)| = 1$, where $\tilde{\mathcal{U}}_{i} = (\tilde{u}_{1,k}, \tilde{u}_{2,k})$. Also note that the function $\tilde{u}_{i,k}$ is defined for all $y$ such that $|y| \leq \varepsilon_{k}^{-1}(s_{k} - d_{g}(x_{0,k} ,x_{1,k}))$. Moreover, if
$
	d_{g}(x,x_{1,k}) \leq r_{k} := (s_{k} - d_{g}(x_{0,k},x_{1,k}))/2
$
	then
$
	d_{g}(x,x_{0,k}) - d_{g}(x_{0,k},x_{1,k}) \leq d_{g}(x,x_{1,k}) \leq r_{k}.
$
This implies that 
$
	d_{g}(x,x_{0,k}) \leq s_{k} + (-s_{k} + d_{g}(x_{0,k}, x_{1,k}))/2.
$
	Thus
$
	r_{k} \leq s_{k} - d_{g}(x,x_{0,k}).
$
	Therefore
	\begin{equation*}
	r_{k}^{\frac{n-2}{2}}|\mathcal{U}(x)|\leq f_{k}(x) \leq f_{k}(x_{1,k}) = (2r_{k}\varepsilon_{k}^{-1})^{\frac{n-2}{2}} = (2r_{k})^{\frac{n-2}{2}}|\mathcal{U}(x_{1,k})|,
	\end{equation*}
that is,
	\begin{equation*}
	|\mathcal{U}(x)| \leq 2^{\frac{n-2}{2}}|\mathcal{U}(x_{1,k})|
	\end{equation*}
	for all $x$ with $d_{g}(x,x_{1,k}) \leq r_{k}$. This implies that 
	\begin{equation*}
	\tilde{u}_{i,k}(y)= |\mathcal{U}(x_{1,k})|^{-1}u_{i}(\exp_{x_{1,k}}(\varepsilon_{k}y)) \leq |\mathcal{U}(x_{1,k})|^{-1}|\mathcal{U}(x)| \leq 2^{\frac{n-2}{2}},
	\end{equation*}
	for all $y$ with $|\varepsilon_{k}y| \leq r_{k}$, that is, $|y| \leq r_{k}\varepsilon_{k}^{-1} \rightarrow \infty.$
	
	Now, if we define $(\tilde{g}_{k})_{lm}(y) := g_{lm}(\varepsilon_{k}y)$, then we have
	\begin{equation}\label{11}
	-\Delta_{\tilde{g}_{k}}\tilde{u}_{i,k}(y) + \varepsilon_{k}^{2}\sum_{j=1}^{2}\tilde{A}^{k}_{ij}(y)\tilde{u}_{j,k}(y) = \frac{n(n-2)}{4}|\tilde{\mathcal{U}_{k}}(y)|^{\frac{4}{n-2}}\tilde{u}_{i,k}(y),
	\end{equation}
	where $\tilde{A}^{k}_{ij}(y):= A_{ij}(\varepsilon_{k})(\varepsilon_{k}y)$.
	Standard elliptic theory then implies that, after passing to a subsequence, that $\{\tilde{u}_{i,k}\}_{k}$ converge in $C^{2}$ norm on compact subsets of $\mathbb{R}^{n}$ to a positive solution $u_{i,0}$ to
	\begin{equation*}
	\Delta u_{i0} + \frac{n(n-2)}{4}|\mathcal{U}_{0}|^{\frac{4}{n-2}}u_{i,0}=0,
	\end{equation*}
	which satisfies $|\mathcal{U}_{0}(0)| = 1$ and $u_{i,0}(y) \leq 2^{\frac{n-2}{2}}$ for every $y \in \mathbb{R}^{n}$. By a theorem due to Druet and Hebey \cite{DH}, we can conclude that there exists $a \in \mathbb{R}^{n}$, $\lambda > 0$ and $\Lambda \in \mathbb{S}^{1}_{+}$ such that
	\begin{equation*}
	\mathcal{U}_{0}(y) = \left(\frac{2\lambda}{1 + \lambda^2|y-a|^{2}}\right)^{\frac{n-2}{2}}\Lambda.
	\end{equation*}
	Since $|\mathcal{U}_{0}(0)| = 1$, we conclude that $|a| \leq 1$ and $\lambda \in [1/2,1]$.
	
	Now, note that $\mathcal{U}_{0}$ has a nondegenerate maximum point at $a$. Then we conclude that there is a sequence $y_{k} \rightarrow a$ such that $y_{k}$ is a nondegenerate maximum point of $|\tilde{\mathcal{U}}_{k}|$. We can assume $|y_{k}| \leq 2$ and therefore there will be a corresponding local maximum point $x_{2,k}$ of $|\mathcal{U}|$ satisfying $d_{g}(x_{2,k},x_{1,k}) \leq 2\varepsilon_{k}$. If we redefine the functions $\tilde{u}_{i,k}$ replacing $x_{1,k}$ by $x_{2,k}$ we get as before that a subsequence $\{\tilde{u}_{i,k}\}_{k}$ converge in the $C^{2}$ norm on compact subsets of $\mathbb{R}^{n}$ to 
	\begin{equation*}
	\mathcal{U}_{0}(y) = \left(\frac{1}{1 + \frac{1}{4}|y|^{2}}\right)^{\frac{n-2}{2}}\Lambda.
	\end{equation*}
	
	Note that, by construction, we have that $|x_{2,k}| < 7/8$, so we can consider $\tilde{u}_{i,k}$ as defined for $|y| \leq \frac{1}{16}\varepsilon_{k}^{-1}$, with a possible singularity at some point on the sphere of radius $|x_{2,k}|\varepsilon_{k}^{-1} \rightarrow \infty$, where now $\varepsilon_{k} = |\mathcal{U}(x_{2,k})|^{-\frac{2}{n-2}}$.
	
	Let us introduce 
	\begin{equation*}
	v_{i,k}(t,\theta) = |y|^{\frac{n-2}{2}}\tilde{u}_{i,k}(y),
	\end{equation*}
	where $t = -\log|y|$ and $\theta = \frac{y}{|y|}$. These functions are defined for $t > \log(\frac{1}{16}\varepsilon_{k}^{-1})$, with a singularity at some point $(t_{k}',\theta_{k}')$, $t_{k}' = \log (|x_{2,k}\varepsilon_{k}^{-1}|)$. Now define,
	\begin{equation*}
	\mathcal{V}_{0}(t) = |y|^{\frac{n-2}{2}}\mathcal{U}_{0}(y) = \left(e^{t} + \frac{1}{4}e^{-t}\right)^{\frac{2-n}{2}}\Lambda.
	\end{equation*}
	Since $\tilde{\mathcal{U}}_{k} \rightarrow \tilde{\mathcal{U}}_{0}$ in the $C^{2}_{loc}$ topology, we know that given $R> 0$ the inequalities 
	\begin{eqnarray}\label{eq015}
	|\mathcal{V}_{k}(t,\theta) - \mathcal{V}_{0}(t)| \leq R^{-1}e^{\frac{2-n}{2}t},\nonumber\\
	|\partial_{t}\mathcal{V}_{k}(t,\theta) - \mathcal{V}_{0}'(t)| \leq R^{-1}e^{\frac{2-n}{2}t}, \nonumber\\
	|\partial_{t}^{2}\mathcal{V}_{k}(t,\theta) - \mathcal{V}_{0}''(t)| \leq R^{-1}e^{\frac{2-n}{2}t},\nonumber\\
	|\partial_{\theta_{l}}\mathcal{V}_{k}(t,\theta)| \leq R^{-1}e^{\frac{2-n}{2}t}\nonumber,\\ 
	|\partial_{\theta_{l}\theta_{m}}^{2}\mathcal{V}_{k}(t,\theta)| \leq R^{-1}e^{\frac{2-n}{2}t}\nonumber
	\end{eqnarray}
	are satisfied for $t > -\log R$ and sufficiently large $k$.
	
	In particular
	\begin{equation}\label{eq016}
	\partial_{t}v_{j,k}(-\log 3n,\theta) = \frac{2-n}{2}\left((3n)^{-1} + \frac{3n}{4}\right)^{\frac{2-n}{2}}\left(\frac{1}{3n}-\frac{3n}{4}\right) - R^{-1}(3n)^{\frac{n-2}{2}} > 0,
	\end{equation}
	for all $\theta \in \mathbb{S}^{n-1}$ and for $R>0$ large enough. 
	
	For $\delta > 0$ a fixed small number to be chosen later, define
	\begin{equation*}
	\Gamma_{k} = [\log(\delta\varepsilon_{k}^{-1}), \infty)\times \mathbb{S}^{n-1}.	
	\end{equation*}
	Since $\tilde{\mathcal{U}}_{k} \rightarrow \mathcal{U}_{0}$ in the $C^{2}_{loc}$ topology and 
	\begin{equation*}
	|\mathcal{V}_{0}(\log(\delta\varepsilon_{k}^{-1}))| = \left(\delta^{-1}\varepsilon_{k} + \frac{\delta\varepsilon_{k}^{-1}}{4}\right)^{\frac{2-n}{2}} \geq c(\delta) > 0,
	\end{equation*}
	then we obtain
	\begin{equation*}
	|\mathcal{V}_{k}(\log(\delta\varepsilon_{k}^{-1}))| \geq c(\delta) > 0.
	\end{equation*}
	
	We will apply the Alexandrov technique to $v_{i,k}$ on the region $\Gamma_{k}$ reflecting across the spheres $\{\lambda\}\times S^{n-1}$. To simplify the notation we will drop the subscript $k$.
	
	Define $\Gamma_{\lambda} = [-\log(\delta\varepsilon^{-1}),\lambda]$, $\hat{g}_{\lambda}$ the pull-back of the metric $\hat{g}$ by the reflection across the sphere $\{\lambda\}\times S^{n-1}$, $v = v_{1} + v_{2}$, $v_{i,\lambda}(t,\theta) = v_{i}(2\lambda - t,\theta)$ and $v_{\lambda}(t) = v(2\lambda -t,\theta)$. 
	
	By $\eqref{mud}$ we can write 
	\begin{equation}\label{10}
	\begin{aligned}
	&\mathcal{L}_{\hat{g}}(v_{i} - v_{i,\lambda}) - \sum_{j=1}^{2}B_{ij}(v_{j} - v_{j,\lambda}) + \frac{n(n-2)}{4}|\mathcal{V}|^{\frac{4}{n-2}}v_{i} - \frac{n(n-2)}{4}|\mathcal{V}_{\lambda}|^{\frac{4}{n-2}}v_{i,\lambda}\\&
	\qquad \qquad= (\mathcal{L}_{\hat{g}_{\lambda}} - \mathcal{L}_{\hat{g}})v_{i,\lambda} - \sum_{j=1}^{2}(B_{ij}^{\lambda} - B_{ij})v_{j,\lambda}.
	\end{aligned}
	\end{equation}
	
	Note that
	\begin{equation*}
	|\mathcal{V}|^{\frac{4}{n-2}}v - |\mathcal{V}_{\lambda}|^{\frac{4}{n-2}}v_{\lambda} =\sum_i^2b_{i\lambda}(v_i-v_{i,\lambda}),
	\end{equation*}
	where
	$$b_{i\lambda}=v(v_i+v_{i, \lambda})\frac{|\mathcal V|^{\frac{4}{n-2}}-|\mathcal V_\lambda|^{\frac{4}{n-2}}}{|\mathcal V|^2-|\mathcal V_\lambda|^2}+|\mathcal V_\lambda|^{\frac{4}{n-2}}>0.$$
	By $\eqref{10}$ we have
	\begin{equation*}
	\mathbb{L}_{\hat{g}}(v_{1} - v_{1,\lambda}, v_{2} - v_{2,\lambda}) + \sum_i^2b_{i\lambda}(v_i-v_{i,\lambda})= Q_{\lambda},
	\end{equation*}
	where
	\begin{equation}\label{eq0091}
	\begin{aligned}
	Q_{\lambda} = (\mathcal{L}_{\hat{g}_{\lambda}} - \mathcal{L}_{\hat{g}})v_{\lambda} + \sum_{i,j=1}^{2}(B_{ij}^{\lambda} - B_{ij})v_{j,\lambda}
	\end{aligned}
	\end{equation}
	and
	\begin{equation}\label{eq009}
	\begin{aligned}
\mathbb{L}_{\hat{g}}(w_{1}, w_{2}) = \mathcal{L}_{\hat{g}}\left(w_{1}+w_2\right) - \sum_{i,j=1}^{2}B_{ij}w_{j}.
	\end{aligned}
	\end{equation}

	\noindent \textbf{Claim 1:} There exists a constant $c_{1}>0$, not depending on $\delta$, such that $|Q_{\lambda}(t,\theta)| \leq q_{\lambda}(t) = c_{1}\varepsilon^{2}e^{\frac{n-6}{2}t}e^{(2-n)\lambda}$.
	
We observe that $\frac{\partial_{r}\sqrt{|g|}}{\sqrt{|g|}} = O(|x|)$ and $v_{\lambda}(t) = O(e^{\frac{2-n}{2}(2\lambda - t)})$. This and \eqref{cur} implies that 
	\begin{equation*}
	\begin{array}{rcl}
	\left|R_{\hat{g}_{\lambda}} - e^{-2t}R_{\Phi^{*}\tilde{g}_{\lambda}} - (R_{\hat{g}} - e^{-2t}R_{\Phi^{*}\tilde{g}})\right|v_{\lambda}(t,\theta) &  \leq & C\varepsilon^{2}e^{-2t}e^{\frac{2-n}{2}(2\lambda -t)}\\ 
	& = & C\varepsilon^{2}e^{\frac{n-6}{2}t}e^{(2-n)\lambda}.
	\end{array}
	\end{equation*}
	Since $\hat{g} = e^{2t}\Phi^{*}\tilde{g}$ and $\tilde{g}_{ij} = \delta_{ij} + O(\varepsilon^{2}|y|^{2})$ in normal coordinates, then $\hat{g} = dt^{2} + d\theta^{2} + O(\varepsilon^{2}e^{-2t})$. It follows that
	\begin{equation*}
	\left|(\Delta_{\hat{g}_{\lambda}} - \Delta_{\hat{g}})v_{\lambda}\right| \leq C\varepsilon^{2}e^{\frac{n-6}{2}t}e^{(2-n)\lambda}.
	\end{equation*}
	Also by $\eqref{eq02}$ we have 
	\begin{equation*}
	\left|\sum_{i,j=1}^{2}(B_{ij}^{\lambda} - B_{ij})v_{j,\lambda}\right| \leq C\varepsilon^{2}e^{\frac{n-6}{2}t}e^{(2-n)\lambda},
	\end{equation*}
	proving the first claim.
	\\
	
	\noindent\textbf{Claim 2: } Suppose $3\leq n\leq 5$, and let $\gamma > 0$ be a small number. Then there exists a family of functions $(h_{1\lambda}(t), h_{2\lambda}(t))$, defined on $\Gamma_{\lambda}$, satisfying the following properties 
	\begin{eqnarray}
	h_{i, \lambda}(\lambda) = 0;\\
	h_{i, \lambda} \geq 0;\\
	\label{aux1}	\mathbb{L}_{\hat{g}}(h_{1, \lambda}, h_{2, \lambda}) \geq Q_{\lambda};\\
	\label{aux}	h_{i, \lambda} \leq v_{i} - v_{i, \lambda}, \mbox{ if } \lambda \mbox{ is sufficiently large}
	\end{eqnarray}
	and 
	\begin{equation}\label{eq004}
	h_{\lambda}(-\log(\delta\varepsilon^{-1})) \leq c_{3}\max\left\{\varepsilon^{\frac{n-2}{2}}, \varepsilon^{\frac{6-n}{2}-\gamma}\right\},
	\end{equation}
	where $h_{\lambda} = h_{1 \lambda}+h_{2 \lambda}$, for some positive constant $c$ which depends only on $\delta$.
	
	Remember that $\hat{g} = e^{\frac{2-n}{2}t}\Phi^{*}\tilde{g}$, then by \eqref{confc} we obtain 
	\begin{equation*}
	L_{\hat{g}}f = e^{-\frac{n+2}{2}t}L_{\Phi^{*}\tilde{g}}(e^{\frac{n-2}{2}t}f) = e^{-\frac{n+2}{2}t}\left(\Delta_{\tilde{g}}(|y|^{\frac{2-n}{2}}\tilde{f}) - \frac{n-2}{4(n-1)}|y|^{\frac{2-n}{2}}R_{\tilde{g}}\tilde{f}\right)\circ \Phi,
	\end{equation*} 
	where $\tilde{f} = f\circ\Phi^{-1}$. Hence
	\begin{equation*}
	\mathcal{L}_{\hat{g}}f = L_{\hat{g}}f + \frac{n-2}{4(n-1)}e^{-2t}R_{\Phi^{*}g}f = e^{-\frac{n+2}{2}t}\Delta_{\tilde{g}}(|y|^{\frac{2-n}{2}}\tilde{f})\circ \Phi.
	\end{equation*}
	So, if $f$ depends only on $t$, then we can use the laplacian to the radial functions to obtain that
	\begin{equation*}
	\mathcal{L}_{\hat{g}}f = f'' + O(\varepsilon^{2}e^{-2t})-f' \left(\left(\frac{n-2}{2}\right)^{2} + O(\varepsilon^{2}e^{-2t})\right)f.
	\end{equation*}
Besides by $\eqref{eq02}$ we obtain
$
	B_{ij} = O(\varepsilon^{2}e^{-2t}).
$
	Thus, if $f =  f_{1}+f_2$, then 
	\begin{equation*}
	\mathbb{L}_{\widehat{g}}(f_{1}, f_{2}) = f'' + O(\varepsilon^{2}e^{-2t})f' -  \left(\left(\frac{n-2}{2}\right)^{2} + O(\varepsilon^{2}e^{-2t})\right)f + \sum_{j=1}^{2}O_{j}(\varepsilon^{2}e^{-2t})f_{j}.
	\end{equation*}
	Given a small number $\gamma >0$, let $\overline{L}$ be the linear operator
	\begin{equation*}
	\overline{L}(f) = f'' + \gamma f' - \left(\left(\frac{n-2}{2}\right)^{2} + \gamma \right)f.
	\end{equation*}
	Let $\gamma_{1} = \frac{8-n}{2}\gamma >0$ and $a(n) = \frac{1}{2(4-n)-\gamma_{1}}$. Define
	\begin{equation}
	h_{i,\lambda}(t) = \frac{a(n)}{2}c_{1}\varepsilon^{2}e^{(2-n)\lambda}e^{\frac{n-6}{2}t}(1-e^{(4-n-\gamma_{2})(t-\lambda)}),
	\end{equation}
	where $\gamma_{2} > 0$ is chosen such that the function $e^{(\frac{2-n}{2}-\gamma_{2})t}$ is in the kernel of $\overline{L}$. Note that $\gamma_{1}$ and $\gamma_{2}$ are small.
	
	Immediately we have $h_{i,\lambda}(\lambda) = 0$, $h_{i, \lambda} \geq 0$ and $h_{i, \lambda}' \leq 0$ in $(-\infty,\lambda]$.
	
	Since $t \geq -\log(\delta\varepsilon^{-1})$, then $\varepsilon^{2}e^{-2t} \leq \delta^{2}$. Choosing $\delta^{2} <<\gamma$, then if $f \geq 0$ and $f_{i}'\leq 0$, we get
	\begin{equation*}
	\mathbb{L}_{\hat{g}}(f_{1}, f_{2}) \geq f'' + \gamma f' -  \left(\left(\frac{n-2}{2}\right)^{2} + \gamma)\right)f = \overline{L}(f).
	\end{equation*}
	Now observe that
 $\overline{L}(h_{\lambda})=q_\lambda$.
	Therefore
	\begin{equation*}
	\mathbb{L}_{\hat{g}}(h_{1,\lambda}, h_{2,\lambda}) \geq \overline{L}(h_{\lambda}) = q_{\lambda} \geq |Q_{\lambda}|.
	\end{equation*}
	Besides,
	\begin{equation}\label{eq005}
	h_{\lambda}(-\log(\delta\varepsilon^{-1})) = a(n)c_{1}\left(\delta^{\frac{6-n}{2}}e^{(2-n)\lambda}\varepsilon^{\frac{n-2}{2}} - \delta^{\frac{n-2}{2} + \gamma_{2}}e^{(-2 + \gamma_{2})\lambda}\varepsilon^{\frac{6-n}{2} - \gamma_{2}}\right),
	\end{equation}
	this give us the estimate (\ref{eq004}).
	
	Since $h_{\lambda}' \leq 0$, then by $\eqref{eq005}$ we obtain that 
	\begin{equation}\label{eq006}
	\max_{\Gamma_{\lambda}}h_{\lambda} \rightarrow 0
	\end{equation}
	as $\lambda \rightarrow \infty$, where $\delta$ and $\varepsilon$ are fixed.
	
	Suppose $\delta$, $\varepsilon$ and $t_{0}$ are fixed with $t_{0}$ large. We know that $v = O(e^{\frac{n-2}{2}t})$. Then, on $\Gamma_{t_{0}} = [-\log(\delta\varepsilon^{-1}),t_{0}]\times \mathbb{S}^{n-1}$ we have, 
	\begin{equation*}
	v_{\lambda}(t) \leq Ce^{\frac{2-n}{2}(2\lambda - t)}.
	\end{equation*}
	Then, by $\eqref{eq006}$ we have that $w_{\lambda} = v - v_{\lambda} - h_{\lambda} \geq 0$ on $\Gamma_{t_{0}}$ for sufficiently large $\lambda$. Let us show that $w_{\lambda} \geq 0$ for all $t \in [t_{0},\lambda]$.
	
	Note that, 
	\begin{equation*}
	h_{\lambda}' = \frac{n-6}{2}a(n)c_{1}\varepsilon e^{(2-n)\lambda}e^{\frac{n-6}{2}t}\left(1 - \frac{2}{n-6}\left(\frac{2-n}{2} - \gamma_{2}\right)e^{(4-n-\gamma_{2})(t-\lambda)}\right).
	\end{equation*}
	Hence, for all $t \in [t_{0},\lambda]$, we have
	\begin{equation*}
	|h_{\lambda}'(t)| \leq C(\delta,\varepsilon,t_{0})(e^{(2-n)\lambda} + e^{(-2+\gamma_{2})\lambda}).
	\end{equation*}
	Now
	\begin{equation}\label{eq007}
	\partial_{t} w_{\lambda}(t,\theta) = \partial_{t} v(t,\theta) + \partial_{t} v(2\lambda -t,\theta) - h_{\lambda}'(t).
	\end{equation}
	Thus, because of the previous estimates and $\eqref{eq015}$, when $t \in [t_{0},\lambda]$, we have that
	\begin{equation*}
	\begin{aligned}
	\partial_{t} w_{\lambda}(t,\theta) \leq -Ce^{\frac{2-n}{2}t} - Ce^{\frac{2-n}{2}(2\lambda -t)} - C(\delta,\varepsilon,t_{0})\left(e^{(2-n)\lambda} + e^{(-2+\gamma_{2})\lambda}\right)\\ \leq -Ce^{\frac{2-n}{\lambda}} < 0,
	\end{aligned}
	\end{equation*}
	for $t_{0}$ and $\lambda$ sufficiently large.
	
	Since $w_{\lambda}(\lambda, \theta) = 0$ for all $\theta \in \mathbb{S}^{n-1}$, then $w_{\lambda} \geq 0$ for all $t \in [t_{0},\lambda]$, for $\lambda$ sufficiently large. This finishes the proof of the Claim 2.
	
Define
	\begin{equation*}
	w_{i, \lambda} = v_i - v_{i, \lambda} - h_{i\lambda} \geq 0.
	\end{equation*}
Note that
	\begin{equation*}
	w_{i, \lambda}(\lambda,\theta) = 0, \mbox{ for all } \theta \in \mathbb{S}^{n-1}
	\end{equation*}
	and
	\begin{equation*}
	\mathbb{L}_{\hat{g}}(w_{1, \lambda}, w_{2, \lambda}) + \sum_i^2b_{i\lambda}w_{i, \lambda} = Q_{\lambda} - \mathbb{L}_{\hat{g}}(h_{1, \lambda}, h_{2, \lambda})\leq 0.
	\end{equation*} 
	
	\noindent\textbf{Claim 3:} There exist $\lambda_{0} > -\log(3n)$ and $\theta_{0} \in \mathbb{S}^{n-1}$ such that 
	\begin{equation*}
	w_{\lambda_{0}}(-\log(\delta\varepsilon^{-1}),\theta_{0}) = 0.
	\end{equation*}
	
	Define
	\begin{equation*}
	\lambda_{0} = \inf\{\lambda_{1}; w_{\lambda}(t,\theta) \geq 0 \mbox{ in } \Gamma_{\lambda}, \forall\lambda\geq\lambda_{1} \}.
	\end{equation*}
	Note that $\eqref{aux}$ implies that this set is not empty. Besides, if we take $\lambda = -\log(3n)$, then by  $\eqref{eq016}$ and $\eqref{eq007}$ we get
	\begin{equation*}
	\partial_{t} w_{\lambda}(-\log(3n),\theta) = 2\partial_{t} v(-\log(3n),\theta) - h'_{\lambda}(-\log(3n)) >0.
	\end{equation*}
	Since $w_{\lambda}(-\log(3n),\theta) = 0$, then $\lambda_{0} > -\log(3n)$.
	
	By continuity, $w_{\lambda_{0}} \geq 0$ in $\Gamma_{\lambda_{0}}$. Suppose the claim is false. Then $w_{\lambda_{0}}(-\log(\delta\varepsilon^{-1}),\theta) > 0$ for all $\theta \in \mathbb{S}^{n-1}$. We notice that by $\eqref{eq02}$, $\eqref{eq009}$ and $\eqref{aux1}$ we can apply the Maximum Principle, since for $\varepsilon > 0$ small enough we have
	\begin{equation}\label{eq018}
	\Delta_{\hat{g}}w_{\lambda_{0}} - Dw_{\lambda_{0}}\leq \mathbb{L}_{\hat{g}}(w_{1\lambda_{0}}, w_{2\lambda_{0}})  =Q_{\lambda_0}-\mathbb{L}_{\hat{g}}(h_{1\lambda_{0}}, h_{2\lambda_{0}}) - \sum_i^2b_{i\lambda}w_{i\lambda_{0}} \leq 0,
	\end{equation}
	where $D$ is a positive function. This implies that $w_{\lambda_{0}}(t,\theta) > 0$ for every $-\log(\delta\varepsilon^{-1}) < t <\lambda_{0}$ and $\theta \in \mathbb{S}^{n-1}$, since on the boundary $\partial\Gamma_{\lambda_{0}}$ we have $w_{\lambda_{0}} \geq 0$.
	
	From definition of $\lambda_{0}$, there exist a sequence $\{\lambda_{j}\}$ such that $\lambda_{j} < \lambda_{0}$ and $\lambda_{j}\rightarrow \lambda_{0}$, a sequence $\{(t_{j},\lambda_{j})\}$ of interior minimum points of $w_{\lambda_{j}}$ such that $(t_{j},\theta_{j})\rightarrow (t^{*},\theta^{*})$ with $w_{\lambda_{j}}(t_{j},\theta_{j}) < 0$. Taking the limit we get $w_{\lambda_{0}}(t^{*},\theta^{*}) = 0$ and $\nabla w_{\lambda_{0}}(t^{*},\theta^{*}) = 0$. Therefore $t^{*} = \lambda_{0}$, but this is a contradiction to the Hopf's lemma. This proves the Claim 3.
	\\
	
	With this claims at hand, let us prove the theorem. By Claim 3, there exist $\lambda_{0} > -\log(3n)$ and $\theta_{0} \in \mathbb{S}^{n-1}$ such that 
	\begin{equation*}
	w_{\lambda_{0}}(-\log(\delta\varepsilon^{-1}),\theta_{0}) = 0.
	\end{equation*}
	Then, by definition of $w_{\lambda_{0}}$ and $\eqref{eq015}$ we get
	\begin{equation*}
	0 < c(\delta) \leq v(-\log(\delta\varepsilon^{-1}),\theta_{0}) = (v_{\lambda_{0}} + h_{\lambda_{0}})(-\log(\delta\varepsilon^{-1}),\theta_{0}).
	\end{equation*}
	But $v(t,\theta) = O(e^{\frac{2-n}{2}(2\lambda - t)})$ implies that
	\begin{equation*}
	v(-\log(\delta\varepsilon^{-1}),\theta_{0}) \leq c(\lambda_{0},\delta)\varepsilon^{\frac{n-2}{2}}
	\end{equation*}
	and so 
	\begin{equation}\label{eq020}
	0 < c(\delta) \leq c(\lambda_{0},\delta)\varepsilon^{\frac{n-2}{2}} + h_{\lambda_{0}}(-\log(\delta\varepsilon^{-1}),\theta_{0}).
	\end{equation}
	
	But by $\eqref{eq004}$ we obtain 
	\begin{equation*}
	h_{\lambda_{0}}(-\log(\delta\varepsilon^{-1}),\theta_{0}) \leq c\varepsilon^{\frac{1}{2}}
	\end{equation*}
	for $n=3$,
	\begin{equation*}
	h_{\lambda_{0}}(-\log(\delta\varepsilon^{-1}),\theta_{0}) \leq c\varepsilon^{1-\gamma}
	\end{equation*}
	for $n=4$ and
	\begin{equation*}
	h_{\lambda_{0}}(-\log(\delta\varepsilon^{-1}),\theta_{0}) \leq c\varepsilon^{\frac{1}{2}-\gamma}
	\end{equation*}
	for $n=5$, which contradicts $\eqref{eq020}$, since we can take the limit $\varepsilon\rightarrow 0$.
	
	This completes the proof of the theorem.
\end{proof}

We notice that in the proof of the Theorem \ref{upper} the condition that $-A$ is cooperative is not needed. Besides that, this result holds for systems with any number of equations.

As a consequence of the upper bound we get the following \emph{spherical Harnack inequality}.

\begin{corollary}\label{DesHarnack}
	Suppose $\mathcal{U}$ is a positive smooth solution of $\eqref{eq000}$ in $\Omega$ with $3\leq n\leq 5$ and that the potential $A$ satisfies the hypotheses \ref{H1}. Then there exists a constant $c_{1} > 0$ such that
	\begin{equation}\label{Harnack}
	\max_{|x|=r} u_{i} \leq c_{1}\min_{|x|=r}u_{i}
	\end{equation} 
	for every $0 < r < \frac{1}{4}$. Moreover, $|\nabla u_{i}| \leq c_{1}|x|^{-1}u_{i}$ and $|\nabla^{2}u_{i}|\leq c_{1}|x|^{-2}u_{i}$.
\end{corollary}

\begin{proof}
	Define $u_{i,r}(y) = r^{\frac{n-2}{2}}u_{i}(ry)$, for every $0 < r < 1/4$ and $|y| < r^{-1}$. Then the upper bound given by Theorem $\ref{upper}$ implies that $u_{i,r}(y) \leq c|y|^{\frac{2-n}{2}}$, for $|y| < \frac{1}{2}r^{-1}$. In particular, if $\frac{1}{2} \leq |y| \leq \frac{3}{2}$, we have that $u_{i,r}(y) \leq 2^{\frac{n-2}{2}}c$.
	
	Moreover
	\begin{equation*}
	\Delta_{g_{r}}u_{i,r}(y) - r^{2}\sum_{j=1}^{2}A_{ij}(ry)u_{j,r}(y) + \frac{n(n-2)}{4}|\mathcal{U}_{r}(y)|^{\frac{4}{n-2}}u_{i,r}(y) = 0.
	\end{equation*}
	where $(g_{r})_{ij}(y) = g_{ij}(ry)$, which implies that
	\begin{equation*}
	\Delta_{g_{r}}u_{i,r}(y) - r^{2}A_{ii}(ry)u_{i,r}(y) = r^{2}A_{ij}(ry)u_{j,r}(y)- \frac{n(n-2)}{4}|\mathcal{U}_{r}(y)|^{\frac{4}{n-2}}u_{i,r}(y).
	\end{equation*}
	where $j\neq i$. Using that $-A$ is cooperative the Harnack inequality for linear elliptic equations and standard elliptic theory imply that there exists $c_{1} > 0$, not depending on $r$, such that
	\begin{equation*}
	\max_{|x|=1}u_{i,r} \leq c_{1} \min_{|x|=1}u_{i,r},
	\end{equation*}
	and $|\nabla u_{i,r}| + |\nabla^{2}u_{i,r}| \leq c_{1}u_{1,r}$ on the sphere of radius 1.
	
	This finishes the proof of the corollary.
\end{proof}

\subsection{Pohozaev invariant and removable singularities}

In this section we will proof a Pohozoaev Identity and define the Pohozaev invariant of a solution of \eqref{eq000}. Also we will prove a removable singularity theorem. As a consequence we will derive a fundamental lower bound near the isolated singularity. 

Given $\mathcal{U}$ a positive solution to the system $\eqref{eq000}$, we define $P(r,\mathcal{U})$ as in $\eqref{P1}$.
The following lemma gives the Pohozaev-type identity we are interested in.
\begin{lemma}[Pohozaev Identity]\label{lemPI} Given $0 < s \leq r < 1$, we have
	\begin{eqnarray*}
		\begin{array}{l}
			\displaystyle P(r,\mathcal{U}) - P(s,\mathcal{U}) = \\
			=\displaystyle-\sum_{i}\int_{B_{r}\backslash B_{s}} \left( x\cdot \nabla u_{i} + \frac{n-2}{2}u_{i}\right)\left((\Delta_{g} - \Delta)u_{i} - \sum_{j}A_{ij}(x)u_{j}\right).
		\end{array}	
	\end{eqnarray*}	
\end{lemma}
\begin{proof} By \eqref{eq000} we get that
	\begin{equation*}
	\begin{array}{l}
	 -\displaystyle\int_{B_{r}\backslash B_{s}} x\cdot \nabla u_{i}\left(\Delta u_{i} + \frac{n(n-2)}{4}|\mathcal{U}|^{\frac{4}{n-2}}u_{i}\right) =\\
	 \displaystyle=\int_{B_{r}\backslash B_{s}}x\cdot \nabla u_{i}\left((\Delta_{g} - \Delta)u_{i} - \sum_{j}A_{ij}(x)u_{j}\right).
	\end{array}
	\end{equation*}
Using integration by parts we get
	\begin{equation}\label{p10}
	\begin{array}{rcl}
	\displaystyle\int_{B_{r}\backslash B_{s}} x\cdot \nabla u_{i}\Delta u_{i} & = &\displaystyle \frac{n-2}{2}\int_{B_{r}\backslash B_{s}} |\nabla u_{i}|^{2} - \frac{r}{2}\int_{\partial B_{r}}|\nabla u_{i}|^{2}  \\
	&&\displaystyle + \frac{s}{2}\int_{\partial B_{s}} |\nabla u_{i}|^{2}+r\int_{\partial B_{r}}\left|\frac{\partial u_{i}}{\partial \nu}\right|^{2} -  r\int_{\partial B_{s}}\left|\frac{\partial u_{i}}{\partial \nu}\right|^{2}. 
	\end{array}
	\end{equation}
	
	On the other hand, integrating over $B_{r}\backslash B_{s}$ we have
	\begin{equation*}
	\begin{array}{l}
\displaystyle \int_{B_{r}\backslash B_{s}} u_{i}\left((\Delta_{g} - \Delta)u_{i} - \sum_{j}A_{ij}(x)u_{j}\right)=\\
\displaystyle= -\int_{B_{r}\backslash B_{s}}u_{i}\left(\Delta u_{i} + \frac{n(n-2)}{4}|\mathcal{U}|^{\frac{4}{n-2}}u_{i}\right)  \\
	\displaystyle=\int_{B_{r}\backslash B_{s}}|\nabla u_{i}|^{2} - \int_{\partial B_{r}}u_{i}\frac{\partial u_{i}}{\partial \nu} + \int_{\partial B_{s}}u_{i}\frac{\partial u_{i}}{\partial \nu} - \frac{n(n-2)}{4}\int_{B_{r}\backslash B_{s}} |\mathcal{U}|^{\frac{4}{n-2}}u_{i}^{2},
	\end{array}
	\end{equation*}
	which implies that 
	\begin{equation*}
	\begin{array}{rcl}
	\displaystyle\int_{B_{r}\backslash B_{s}}|\nabla u_{i}|^{2} & = &\displaystyle \int_{\partial B_{r}}u_{i}\frac{\partial u_{i}}{\partial \nu} - \int_{\partial B_{s}}u_{i}\frac{\partial u_{i}}{\partial \nu} + \frac{n(n-2)}{4}\int_{B_{r}\backslash B_{s}} |\mathcal{U}|^{\frac{4}{n-2}}u_{i}^{2}\\
	&&\displaystyle + \int_{B_{r}\backslash B_{s}} u_{i}\left((\Delta_{g} - \Delta)u_{i} - \sum_{j}A_{ij}(x)u_{j}\right).
	\end{array}
	\end{equation*}
Thus, by \eqref{p10} we conclude that
	\begin{equation}\label{p1}
	\begin{aligned}
	&\sum_{i}\int_{B_{r}\backslash B_{s}} \left(x\cdot \nabla u_{i} + \frac{n-2}{2}u_{i}\right)\left((\Delta_{g} - \Delta)u_{i} - \sum_{j}A_{ij}(x)u_{j}\right) \\
	&=- \frac{n(n-2)^2}{8}\int_{B_{r}\backslash B_{s}}|\mathcal{U}|^{\frac{2n}{n-2}}  - \frac{n(n-2)}{4}\sum_{i}\int_{B_{r}\backslash B_{s}}(x\cdot \nabla u_{i}|\mathcal{U}|^{\frac{4}{n-2}}u_{i}).
	\end{aligned}
	\end{equation}

Using that $\frac{n-2}{2n}\partial_{k}(|\mathcal{U}|^{\frac{2n}{n-2}}) = \sum_{i}|\mathcal{U}|^{\frac{4}{n-2}}u_{i}\partial_{k}u_{i}$ we obtain
	\begin{equation*}
	\begin{array}{rcl}
	\displaystyle- \sum_{i}\int_{B_{r}\backslash B_{s}}x\cdot \nabla u_{i}|\mathcal{U}|^{\frac{4}{n-2}}u_{i} & = & 	\displaystyle\frac{n-2}{2}\int_{B_{r}\backslash B_{s}}|\mathcal{U}|^{\frac{2n}{n-2}} - r\frac{n-2}{2n}\int_{\partial B_{r}} |\mathcal{U}|^{\frac{2n}{n-2}} \\
	& &	\displaystyle+ s\frac{n-2}{2n}\int_{\partial B_{s}} |\mathcal{U}|^{\frac{2n}{n-2}}
	\end{array}
	\end{equation*}
	and consequently, replacing in \eqref{p1}, we conclude the result.
	
\end{proof}

In the case of the limit system we saw in $\eqref{pohid}$ that $P(r,\mathcal{U})$ does not depend on $r$, therefore is an invariant of the solution $\mathcal{U}$.

Motivated by \cite{marques} we define the Pohozaev invariant of a solution $\mathcal U$ of \eqref{eq000} by using the Theorem \ref{upper}. Since $g_{ij} = \delta_{ij} + O(|x|^{2})$, then we get that
\begin{equation}\label{23}
\left|\sum_i\left( x\cdot \nabla u_{i} + \frac{n-2}{2}u_{i}\right)\left((\Delta - \Delta_{g})u_{i} - \sum_{j}A_{ij}(x)u_{j}\right)\right| \leq c|x|^{2-n},
\end{equation}
for some positive constant $c$ which does not depend on $x$. The Pohozaev identity tell us the limit 
\begin{equation*}
P(\mathcal{U}) := \lim_{r\rightarrow 0}P(r,\mathcal{U})
\end{equation*}
exists. The number $P(\mathcal{U})$ is called the \textit{Pohozaev invariant} of the solution $\mathcal{U}$.

Once defined the Pohozaev invariant, our main result of this section reads as follows.
\begin{theorem}\label{lower} Let $\mathcal{U}$ be a positive solution to the system \eqref{S} in $B^{n}_{1}(0)\backslash\{0\}$ such that  $3\leq n\leq 5$ and the potential A satisfies the hypotheses \ref{H1} and \ref{H2}. Then $P(\mathcal{U}) \leq 0$. Moreover, $P(\mathcal{U}) = 0$ if and only if each coordinate $u_{i}$ is smooth on the origin.
\end{theorem}

The strategy to reach the result is to assume that $P(\mathcal{U}) \geq 0$ and proof that the origin is a removable singularity and $P(\mathcal{U}) =0$. In what follows let us denote by $$u(x) = u_{1}(x)+ u_{2}(x),$$ $\overline{u}$ the average of $u$ over $\partial B_{r}$, that is,
$$\overline u(r):=\fint_{\partial B_{r}} u:=\frac{1}{vol(\partial B_r)}\int_{\partial B_{r}} u,$$
and define $$w(t) = \overline{u}(r)r^{\frac{n-2}{2}},$$ where $t = -\log r$.

We have divided the proof into a sequence of lemmas.

\begin{lemma}\label{n0} Let $\mathcal{U}$ be a positive solution of $\eqref{eq000}$ which satisfies $P(\mathcal{U}) \geq 0$. Then
	\begin{equation*}
	\liminf_{x\rightarrow 0} u(x)|x|^{\frac{n-2}{2}} = 0.
	\end{equation*}
\end{lemma}
\begin{proof}
	If this result is not true, without loss of generality we can suppose that there exist positive constants $c_{1}$ and $c_{2}$ such that
	\begin{equation}\label{io}
	c_{1}|x|^{\frac{2-n}{2}}\leq u_{1}(x) \leq c_{2}|x|^{\frac{2-n}{2}},
	\end{equation}
	where the second inequality above follows from Theorem \ref{upper}. Choose any sequence $r_{k} \rightarrow 0$, and define
	\begin{equation*}
	u_{i,k}(x) = r_{k}^{\frac{n-2}{2}}u_{i}(r_{k}x).
	\end{equation*}
	Then, using \eqref{io}, we have 
	\begin{equation}\label{io1}
	c_{1}|x|^{\frac{2-n}{2}}\leq u_{1,k}(x) \leq  c_{2}|x|^{\frac{2-n}{2}}.
	\end{equation}
	Moreover  $u_{i,k}$ satisfies 
	\begin{equation*}
	-\Delta_{g_{k}}u_{i,k} + r_{k}^{2}\sum_{j=1}^{2}A_{ij}(r_{k}x)u_{j,k} = \frac{n(n-2)}{4}|\mathcal{U}_{k}|^{\frac{4}{n-2}}u_{i,k}\quad \mbox{in}\quad B_{r_{k}^{-1}}(0)\backslash\{0\},
	\end{equation*} 
	where $(g_{k})_{lm}(x) = g_{lm}(r_{k}x)$. Elliptic theory then implies that there exists a subsequence, also denoted by $u_{i,k}$, which converges in compact subsets of $\mathbb{R}^{n}\backslash \{0\}$ to a solution $\mathcal{U}_{0} = (u_{1,0}, u_{2,0})$ of the limit system
	\begin{equation*}
	\Delta u_{i,0} + \frac{n(n-2)}{4}|\mathcal{U}_{0}|^{\frac{4}{n-2}}u_{i,0} = 0.
	\end{equation*}
	By \eqref{io1} we get 
	\[
	u_{1,0}(x) \geq c_{1}|x|^{\frac{2-n}{2}},
	\]
	which implies that $\mathcal{U}_{0}$ is singular at the origin. However, by Theorem \ref{class} and Corollary \ref{classical} we know that $P(\mathcal{U}_{0}) < 0$.  This is a contradiction, because
$$
	P(\mathcal{U}_{0}) = P(1, \mathcal{U}_{0}) = \lim_{k\rightarrow \infty} P(1, \mathcal{U}_{k}) = \lim_{k\rightarrow \infty} P(r_{k},\mathcal{U}) = P(\mathcal{U}) \geq 0.
$$
\end{proof}
\begin{lemma}\label{n1} Assume that $\mathcal{U}$ is a positive solution of $\eqref{eq000}$ and that the potential A satisfies the hypotheses \ref{H1} and \ref{H2}. Suppose that there exists a sequence $(t_k)$ of minimum points for $w$ such that $\displaystyle\lim_{k\rightarrow \infty}w(t_k) = 0$, where $w=w_1+w_2$, $w_i=r^{\frac{n-2}{2}}\overline u_i$ and $r=e^{-t}$. Along $|x| = r_{k}$ there exist positive constants $a_{i}$ and $b_{i}$ such that
\begin{equation}\label{bo}
\begin{array}{rcl}
	u_{i}(x) & = & \overline{u}(r_{k})(c_{i} + o(1)) \\ 
	|\nabla u_{i}(x)| & = & -\overline{u}'(r_{k})(2a_{i}+o(1)).
\end{array}
\end{equation}
	where $c_{i} = a_{i} + b_{i}$ and $t_k = -\ln r_k$. Moreover, the vectors $a$ and $b$ are not orthogonal.
\end{lemma}
\begin{proof}  Define $r_{k} = e^{-t_{k}}$ and $v_{i,k}(y) = r_{k}^{\frac{n-2}{2}}u_{i}(r_{k}y).$
Since $\overline{v}_{i,k}(1) = w_{1}(t_{k}) \leq w_{1}(t_{k}) + w_{2}(t_{k}) = w(t_{k}) \rightarrow 0$ we get from Harnack inequality that each coordinate $v_{i,k}$ converge uniformly in compact subsets of $\mathbb{R}^{n}\backslash\{0\}$ to $0$. So, if we define
	\begin{equation*}
	h_{i,k}(y) = v_{k}(p)^{-1}v_{i,k}(y),
	\end{equation*} 
	where $v_{k} = \sum_{i}v_{i,k}$ and $p = (1,0,\ldots,0)$, we have that
	\begin{eqnarray*}
		-\Delta_{g_{k}}h_{i,k}(y) &=& -v_{k}(p)^{-1}\left(\Delta_{g_{k}}v_{i,k}\right)\\
		&=& v_{k}(p)^{-1}\left(\frac{n(n-2)}{4}|\mathcal{V}_{k}|^{\frac{4}{n-2}}v_{i,k} - r_{k}^{2}\sum_{j=1}^{2}A_{ij}(r_{k}y)v_{j,k}\right).
	\end{eqnarray*}
	This implies that
	\begin{equation*}
	-\Delta_{g_{k}}h_{i,k}  + r_{k}^{2}\sum_{j=1}^{2}A_{ij}(r_{k}y)h_{j,k} = \frac{n(n-2)}{4}v_{k}(p)^{\frac{4}{n-2}}|H_{k}|^{\frac{4}{n-2}}h_{i,k},
	\end{equation*}
	where $(g_{i})_{lm}(y) = g_{lm}(r_{i}y)$ and $H_k=(h_{1,k},h_{2,k})$.
	By elliptic estimates we know that there exists a subsequence $h_{i,k}$ which converge in $C^{2}_{loc}$ to a nonnegative harmonic function $h_{i}$ in $\mathbb{R}^{n}\backslash\{0\}$. Then $$h_{i}(y) = a_{i}|y|^{2-n} + b_{i},$$ and $a_{1} + a_{2} = b_{1} + b_{2} = 1/2$, since $h_{1}(p) + h_{2}(p) = 1$ and $\partial_{r}((h_{1} + h_{2})(r)r^{\frac{n-2}{2}}) = 0$ at $r =1$.  Note that if $|y|=1$, then $h_i(y)=a_i+b_i$. This implies that $h_{i,k}(y)=c_i+o(1)$. From this we get the first equality in \eqref{bo}. Analogously, we obtain the second equality.

Since
$$-\Delta_{g_k}v_{i,k}+r_k^2\sum_j\tilde A_{ij}v_{jk}=\frac{n(n-2)}{4}|V_k|^{\frac{4}{n-2}}v_{ik},$$
where $\tilde A_{ij}(y)=A_{ij}(r_ky)$, then we get that
	\begin{equation*}
	\int_{B_1\backslash B_{\varepsilon}} (v_{i,k}\Delta_{g}v_{j,k} - v_{j,k}\Delta_{g}v_{i,k})dv_{g} = r_k^2\sum_{l=1}^{2}\int_{B_1\backslash B_{\varepsilon}} (v_{j,k}\tilde A_{il} - v_{i,k}\tilde A_{jl})v_{l,k}dv_{g}.
	\end{equation*}
Integrating by parts, we obtain that
	\begin{equation}\label{ast}
	\begin{array}{rcl}
		\displaystyle\int_{\partial B_{1}}(v_{i,k}\partial_{r}v_{j,k} - v_{j,k}\partial_{r}v_{i,k})d\sigma_{g} & = & r_k^2\displaystyle\sum_{l=1}^{2}\int_{B_1\backslash B_{\varepsilon }} (v_{j,k}\tilde A_{il} - v_{i,k}\tilde A_{jl})v_{l,k}dv_{g}\\
		& &  + \displaystyle\int_{\partial B_{\varepsilon}}(v_{i,k}\partial_{r}v_{j,k} - v_{j,k}\partial_{r}v_{i,k})d\sigma_{g}.
	\end{array}
	\end{equation}
	
	In order to analyse the last integral on the right-hand side of \eqref{ast}, define 
$
	\varphi_{i,k}^{\varepsilon}(z) = \varepsilon^{\frac{n-2}{2}}v_{i,k}(\varepsilon z).
$
	Then
	\begin{equation*}
	\int_{\partial B_{\varepsilon}}(v_{i,k}\partial_{r}u_{j,k} - u_{j,k}\partial_{r}u_{i,k})d\sigma_{g} = \int_{\partial B_{1}} (\varphi_{i,k}^{\varepsilon}\partial_{r}\varphi_{j,k}^{\varepsilon} - \varphi_{j,k}^{\varepsilon}\partial_{r}\varphi^{\varepsilon}_{j,k})d\sigma_{g}.
	\end{equation*}
	On the other hand
	\begin{equation*}
	-\Delta_{g_{\varepsilon k}}\varphi_{i,k}^{\varepsilon} + (\varepsilon r_{k})^{2}\sum_{j=1}^{2}\tilde{A}^{\varepsilon}_{ij}\varphi_{j,k}^{\varepsilon} = c(n)|\varphi_{k}^{\varepsilon}|^{\frac{4}{n-2}}\varphi_{i,k}^{\varepsilon}
	\end{equation*}
	in $B_{(\varepsilon r_{k})^{-1}}(0)\backslash \{0\}$ and by Theorem \ref{upper}
	\begin{equation*}
	|\varphi^{\varepsilon}_{k}(z)| \leq C|z|^{\frac{2-n}{2}}.
	\end{equation*}
	Similarly to what we did at the beginning of this proof, after passing to a subsequence, $\varphi_{i,k}^{\varepsilon}$ converges in $C^{2}$ topollogy locally in compact subsets of $\mathbb{R}^{n} \backslash \{0\}$, to a positive solution of
	\begin{equation*}
	\Delta u_i + \frac{n(n-2)}{4}|\mathcal{U}|^{\frac{4}{n-2}}u_i = 0
	\end{equation*}
	which using Theorem \ref{class} as well Proposition 1.1 in \cite{DHV}, is of the form $u_{0}\Lambda$, where $\Lambda =(\Lambda_1, \Lambda_2)$ is a vector in the unit sphere with nonnegative coordinates. Consequently, when $\varepsilon$ goes to zero we have
	\begin{equation*}
	\lim_{\varepsilon \to 0}\int_{\partial B_{1}} (\varphi_{j,k}^{\varepsilon}\partial_{r}\varphi_{i,k}^{\varepsilon} - \varphi_{i,k}^{\varepsilon}\partial_{r}\varphi^{\varepsilon}_{i,k})d\sigma_{g} = 	\int_{\partial B_{1}} (u_{0}\partial _{r}u_{0}\Lambda_{i}\Lambda_{j} - u_{0}\partial _{r}u_{0}\Lambda_{j}\Lambda_{i})d\sigma_{g} = 0.
	\end{equation*}
Thus by \eqref{ast}, we conclude that 
	\begin{equation}
			\displaystyle\int_{\partial B_{1}}(v_{i,k}\partial_{r}v_{j,k} - v_{j,k}\partial_{r}v_{i,k})d\sigma_{g} = r_k^2\displaystyle\sum_{l=1}^{2}\int_{B_1\backslash B_{\varepsilon }} (v_{j,k}\tilde A_{il} - v_{i,k}\tilde A_{jl})v_{l,k}dv_{g}.
	\end{equation}
Since $|v_{i,k}|\leq c|x|^{\frac{2-n}{2}}$, then the integral in the right hand side above is convergent and it is uniformly bounded in $k$.

This implies that
	\begin{align*}
	&\int_{\partial B_{1}}(h_{j,k}\partial_{r}h_{i,k} -  h_{i,k}\partial_{r}h_{j,k})d\sigma_{g} \\
	&\qquad\qquad \qquad = r_{k}^{4-n}u(r_kp)^{-2} \sum_{l=1}^{2}\int_{B_{1}\backslash\{0\}} (v_{j,k}\tilde A_{il}-v_{i,k}\tilde A_{jl})v_{l,k}dv_{g}.
	\end{align*}
Since $\displaystyle\lim_{k\rightarrow\infty}u(r_kp)=\infty$ and the integral in hand right side is uniformly bounded and $A$ satisfies the hypotheses \ref{H2}, we conclude that 
	\begin{equation*}
	\int_{\partial B_{1}}(h_{j}\partial_{r}h_{i} - h_{i}\partial_{r}h_{j})d\sigma_{g} = 0,
	\end{equation*}
	and consequently
	\begin{equation*}
	a_1b_2 = a_2b_1
	\end{equation*}
	which finishes the proof.\end{proof}

\begin{lemma}\label{n2} 	Assume that $\mathcal{U}$ is a positive solution of $\eqref{eq000}$ and along $|x| = r_k$ it satisfies $\eqref{bo}$. Then 
\[	
	P(r_{k},\mathcal{U}) = \sigma_{n-1}\left( - 4\langle a,b\rangle\frac{\delta^2}{2} w^{2}(t_{k}) + \frac{\delta^2}{2}(c_{1}^2 + c_{2}^2)^{\frac{n}{n-2}}|W|^{\frac{2n}{n-2}}(t_{k})\right)(1 + o(1)),
\]
		where $w$ is defined in Lemma \ref{n1}.
\end{lemma}

\begin{proof}
	Since the solution $\mathcal{U}$ satisfies $\eqref{bo}$, note that 
	\begin{equation*}
	\left\langle \mathcal{U}, \frac{\partial \mathcal{U}}{\partial \nu} \right\rangle = \sum_{i=1}^{2} \bar{u}(r_k)\bar{u}'(r_k)(2c_i a_i + o(1)).
	\end{equation*}
	On the other hand, using that $w_{t}(t_k) = 0$ it holds the following equality
	\begin{equation}\label{m1}
	\bar{u}'(r_k)r_{k}^{\frac{n}{2}} = -\frac{n-2}{2}w(t_k).
	\end{equation}
	Multiplying $\eqref{m1}$ by $\bar{u}(r_k)$, we conclude that
	\begin{equation*}
	\frac{n-2}{2}\left\langle \mathcal{U}, \frac{\partial \mathcal{U}}{\partial \nu} \right\rangle = -\sum_{i=1}^{2}  r_k^{1-n}\frac{(n-2)^2}{2}w^{2}(t_k)(c_i a_i + o(1)).
	\end{equation*}
	Similarly, we also have
	\begin{align*}
	-\frac{r}{2}|\nabla \mathcal{U}|^2 + r\left|\frac{\partial \mathcal{U}}{\partial\nu}\right|^{2} &= r_k\sum_{i=1}^{2}\bar{u}'(r_k)^2(2a_i^2 + o(1))\\
	&=  r_{k}^{1-n}\sum_i\frac{(n-2)^2}{2}w^{2}(t_{k})(a_i^2 + o(1)).
	\end{align*}
	With these two equalities we get
	\begin{equation*}
	\begin{array}{l}
	\displaystyle P(r_k,\mathcal{U})  = \\
 = \displaystyle \int_{\partial B_{r_{k}}}\left(\frac{n-2}{2}\left\langle \mathcal{U}, \frac{\partial \mathcal{U}}{\partial \nu} \right\rangle - \frac{r}{2}|\nabla \mathcal{U}|^{2}+ r\left|\frac{\partial\mathcal{U}}{\partial\nu}\right|^{2} \right.\displaystyle\left.+  r\frac{(n-2)^2}{8}|\mathcal{U}|^{\frac{2n}{n-2}}\right)d\sigma\\
	 =  \displaystyle\sigma_{n-1}\left( - 4\langle a, b\rangle\frac{\delta^2}{2}w(t_k)^2 + \frac{\delta^2}{2}(c_1^2 + c_2^2)^{\frac{n}{n-2}}|W|^{\frac{2n}{n-2}}(t_k)\right)(1+o(1)),
	\end{array}
	\end{equation*}
	which finishes the proof.
\end{proof}

\begin{lemma}\label{l2} Let $\mathcal{U}$ be a positive solution of $\eqref{eq000}$ defined in the punctured ball. If
	\begin{equation}\label{hi}
	\lim_{|x|\rightarrow 0}|x|^{\frac{n-2}{2}}u(x) = 0,
	\end{equation}
	then $\mathcal{U}$ extends as a smooth solution to all unit ball $B^n$.
\end{lemma}

\begin{proof}
	We begin by obtaining upper and lower bounds for the second derivatives of $w$ in terms of $w$, where $w$ is defined in Lemma \ref{n1}. Indeed, observe that the upper bound in Theorem \ref{upper} implies that $w(t)$ is bounded. Then,
	\begin{equation*}
	\overline{u}_{r} = \fint_{\partial B_{r}} u_{r},
	\end{equation*}
	and since
	\begin{equation*}
	w_t = - \overline{u}_{r}r^{\frac{n}{2}} - \frac{n-2}{2}w,
	\end{equation*}
	we also get that $|w_t|$ is bounded. Derivating again the function $w$ we obtain 
	\begin{equation}\label{w}
	w_{tt} = \overline{u}_{rr}r^{\frac{n+2}{2}} + (n-1)\overline{u}_{r}r^{\frac{n}{2}} + \left(\frac{n-2}{2}\right)^{2}w.
	\end{equation}	
	Choosing a fixed $s < r$, by the Divergence Theorem, we get
	$$\begin{array}{rcl}
	  \displaystyle\left(\int_{B_{r}\backslash B_{s}}\Delta u\right)_r & = &\displaystyle\left(\int_{\partial B_r}u_r\right)_r=\left(\int_{\partial B_1}u_r(r\cdot)r^{n-1}\right)_r\\
& = &\displaystyle\frac{n-1}{r}\int_{\partial B_1}u_r(r\cdot)r^{n-1}+\int_{\partial B_1}u_{rr}(r\cdot)r^{n-1} \\
& = &\displaystyle\frac{n-1}{r}\int_{\partial B_r}u_r+r^{n-1}\sigma_{n-1}\left(\fint_{\partial B_r}u_{r}\right)_r,	  
	  \end{array}
$$
	where $\sigma_{n-1}=vol(\mathbb S^{n-1})$. Thus	
	\begin{equation*}
	\begin{aligned}
	&\overline{u}_{rr} = \left(\fint_{\partial B_{r}} u_{r}\right)_{r} = (1-n)r^{-1}\fint_{\partial B_{r}}u_{r} + \sigma_{n-1}^{-1}r^{1-n}\left(\int_{B_{r}\backslash B_{s}}\Delta u\right)_{r}\\
	&= \frac{1-n}{r}\overline{u}_{r}+ \frac{r^{1-n}}{\sigma_{n-1}}\sum_{i=1}^{2}\left(\int_{B_{r}\backslash B_{s}}\left((\Delta - \Delta_{g})u_{i} + \sum_{j=1}^{2}A_{ij}u_{j} - c(n)|\mathcal{U}|^{\frac{4}{n-2}}u_{i}\right)\right)_{r}
	\end{aligned}
	\end{equation*}
	and so
	$$\overline{u}_{rr}+ \frac{n-1}{r}\overline{u}_{r} = \fint_{\partial B_{r}}\left((\Delta -\Delta_{g})u + 	\sum_{i,j=1}^{2}A_{ij}(x)u_{j} - c(n)|\mathcal{U}|^{\frac{4}{n-2}}u\right),$$
	where $c(n)=\frac{n(n-2)}{4}$. Replacing in \eqref{w} we have that 
	\begin{equation*}
	\begin{aligned}
	&w_{tt} - \left(\frac{n-2}{2}\right)^{2}w = r^{\frac{n+2}{2}}\fint_{\partial B_{r}}\left((\Delta - \Delta_{g})u + \sum_{i,j=1}^{2}A_{ij}(x)u_{j} -  c(n)|\mathcal{U}|^{\frac{4}{n-2}}u\right).
	\end{aligned}
	\end{equation*}
	
	Applying the spherical Harnack inequality obtained in Corollary \ref{DesHarnack} for each coordinate function we have that 
	\begin{equation*}
	C^{-1}\overline{u}^{\frac{n+2}{n-2}} \leq c(n) \fint_{\partial B_{r}} |\mathcal{U}|^{\frac{4}{n-2}}u \leq C\overline{u}^{\frac{n+2}{n-2}}
	\end{equation*}
	and
	\begin{equation*}
	\left|\fint_{\partial B_{r}}\left((\Delta - \Delta_{g})u + \sum_{i,j=1}^{2}A_{ij}(x)u_{j}\right)\right| \leq c\overline{u}.
	\end{equation*}
	With these estimates we obtain the following inequality
	\begin{equation}\label{24}
	-c_{1}w^{\frac{n+2}{n-2}} - c_{3}e^{-2t}w \leq w_{tt} - \left(\frac{n-2}{2}\right)^{2}w \leq -c_{2}w^{\frac{n+2}{n-2}} + c_{3}e^{-2t}w.
	\end{equation}
	By hypotesis $\eqref{hi}$ we know that $\lim_{t\rightarrow \infty}w(t) =0$. The strategy is to show that $u \in L^{p}_{loc}(B^{n}_{1}(0))$ for some $p>2n/(n-2)$ and then by elliptic theory the function $u$ extends smoothly across the origin. Consequently, each coordinate function $u_{i}$ is smooth.
	
	Note that the first inequality in $\eqref{24}$ implies that there exists $\varepsilon_{0} > 0$, such that if $w(t) \leq \varepsilon_{0}$ and $t$ is sufficiently large, then $w_{tt}(t) > 0$. Since $\displaystyle\lim_{t\rightarrow \infty}w(t) =0$, there exists $T_1$ so that $w(t)<\varepsilon_0$ and $w_{tt} > 0$ for $t \geq T_{1}$. This implies that $w_t < 0$ for $t \geq T_{1}$.
	
	By the first inequality in $\eqref{24}$, given any positive number $0<\delta<n-2$, there exists $T_{0}$ sufficiently large such that
	\begin{equation*}
	w_{tt} - \left(\frac{n-2}{2}- \delta\right)^{2}w \geq (\delta(n-2)-\delta^2-c_1w^{\frac{4}{n-2}}-c_3e^{-2t})w \geq 0,
	\end{equation*}
	for $t \geq T_{0}$, which implies that
	\begin{equation*}
	\left(w_{t}^{2} - \left(\frac{n-2}{2} - \delta\right)^{2}w^2\right)_{t} =  2w_{t}\left( w_{tt} - \left(\frac{n-2}{2} - \delta\right)^{2}w\right)\leq 0,
	\end{equation*}
	for $t>T_2=\max\{T_0,T_1\}$, and using that $\displaystyle\lim_{t \rightarrow \infty}w_{t}(t) = 0$, we obtain
	\begin{equation*}
	w_{t}^{2} - \left(\frac{n-2}{2} - \delta\right)^{2}w^2 \geq 0.
	\end{equation*}
	By integrating we get, for $t \geq T_{2}$, that 
	\begin{equation*}
	w(t) \leq w(T_{0})e^{-\left(\frac{n-2}{2} - \delta\right)(t-T_{0})}.
	\end{equation*}
	Equivalently, there exists $r_{0}(\delta) > 0$, so that
	\begin{equation*}
	u(x) \leq c(\delta)|x|^{-\delta} \mbox{ for all } x \in B_{r_0}(0).
	\end{equation*}
	
	Since $\delta > 0$ is arbitrarily small, the estimate above implies that $u \in L^{p}_{loc}(B_{1}(0))$ for arbitrarily large $p$, which finishes our proof.
\end{proof}

Now we are ready to prove the main result of this section.

\begin{proof}[Proof of Theorem \ref{lower}] Following the aforementioned strategy let us suppose that $P(\mathcal{U}) \geq 0$. By Lemma \ref{l2} the proof is completed by showing that $\displaystyle\lim_{x\rightarrow 0}u(x)|x|^{\frac{n-2}{2}} = 0$.  
	Suppose by contradiction that this is false. 
	
	Since the Pohozaev invariant is nonnegative, it holds as a consequence of Lemma \ref{n0} that $\displaystyle\liminf_{x\rightarrow 0}u(x)|x|^{\frac{n-2}{2}} = 0$. Then we will assume that $\displaystyle\limsup_{x\rightarrow 0}u(x)|x|^{\frac{n-2}{2}} > 0$. Hence we can choose $\varepsilon_{0} > 0$ sufficiently small so that we are able to construct sequences $\bar{t}_{k} \leq t_{k} \leq t_{k}^{*}$ satisfying $\displaystyle\lim_{k\rightarrow \infty} \bar{t}_{k} = +\infty$, $w(\bar{t}_{k}) = w(t_{k}^{*}) = \varepsilon_{0}$, $w_{t}(t_{k}) = 0$ and $\displaystyle\lim_{k \rightarrow \infty}w(t_{k}) = 0$, where $w$ is defined in Lemma \ref{n1}.
	
	Using Lemma~\ref{n2}, we have that 
\[
	P(r_{k},\mathcal{U}) = \sigma_{n-1}\left( - 4\langle a,b\rangle\frac{\delta^2}{2} w^{2}(t_{k}) + \frac{\delta^2}{2}(c_{1}^2 + c_{2}^2)^{\frac{n}{n-2}}|W|^{\frac{2n}{n-2}}(t_{k})\right)(1 + o(1)),
\]
	and consequently
	\begin{equation*}
	P(\mathcal{U}) = \lim_{k\rightarrow \infty} P(r_{k},\mathcal{U}) = 0.
	\end{equation*}
	Morevover, by Lemma \ref{n1} the vectors $a$ and $b$ are not orthogonal. Thus
	\begin{equation}\label{30}
	w^{2}(t_{k}) \leq c|P(r_{k},\mathcal{U})| \leq c(I_{1} + I_{2}),
	\end{equation}
	where
	\begin{equation*}
	I_{1} = \int_{B_{r_{k}}\backslash B_{r_{k}^{*}}} |A(\mathcal{U})|dx,
	\end{equation*}
	
	\begin{equation*}
	I_{2} = \int_{B_{r_{k}^{*}}} |A(\mathcal{U})|dx,
	\end{equation*}
	and $A(\mathcal{U})  = \sum_{i} \left( x\cdot \nabla u_{i} + \frac{n-2}{2}u_{i}\right)\left((\Delta_{g} - \Delta)(u_{i}) - \sum_{j}A_{ij}(x)u_{j}\right)$.
	We will follow a series of calculations with the goal of obtaining better estimates for the term $I_1$. 
	It follows by the first inequality of $\eqref{24}$ that 
	\begin{equation*}
	w_{tt} - \left(\frac{n-2}{2}\right)^{2}w \geq -c_{1}w^{\frac{n+2}{n-2}} - c_{3}e^{-2t}w \geq -c_{1}w^{\frac{n+2}{n-2}} - c_{3}e^{-2t_{k}}w
	\end{equation*}
	for $t \geq t_{k}$, which implies 
	\begin{equation*}
	\frac{d}{dt}\left[ w_{t}^{2} - \left(\left(\frac{(n-2)}{2}\right)^{2} - c_{3}e^{-2t_{k}}\right)w^{2} + \frac{n-2}{n}c_{1}w^{\frac{2n}{n-2}}\right] \geq 0
	\end{equation*}
	for $t_{k} \leq t \leq t_{k}^{*}$.
	
	Hence, if $t_{k} \leq t \leq t_{k}^{*}$, then
	\begin{equation*}
	w_t(t)^2-g(w(t))+g(w(t_k))=\int_{t_{k}}^{t}\frac{d}{dt}\left(w_t^{2} - g(w)\right) \geq 0,
	\end{equation*}
	which implies that
	$$\frac{dw}{dt}\geq \sqrt{g(w)-g(w(t_k))}$$
	and so
	\begin{equation*}
	t - t_{k} = \int_{w(t_{k})}^{w(t)}\frac{dt}{dw}dw \leq \int_{w(t_{k})}^{w(t)} \frac{dw}{\sqrt{g(w) - g(w(t_{k}))}},
	\end{equation*}
	where 
	\begin{equation*}
	g(w) = \left(\left(\frac{(n-2)}{2}\right)^{2} - c_{3}e^{-2t_{k}}\right)w^{2} - \frac{n-2}{n}c_{1}w^{\frac{2n}{n-2}}.
	\end{equation*}
	
	Introducing the variable $\eta = \frac{w(t)}{w(t_{k})}$, we get
	\begin{equation}\label{25}
	t - t_{k} \leq \int_{1}^{\frac{w(t)}{w(t_{k})}} \frac{d\eta}{\sqrt{\overline{g}(\eta) - \overline{g}(1)}} = \int_{1}^{\frac{w(t)}{w(t_{k})}} \sqrt{\frac{\eta^{2} -1}{\overline{g}(\eta) - \overline{g}(1)}}\frac{d\eta}{\sqrt{\eta^{2} -1}},
	\end{equation}
	where
	\begin{equation*}
	\overline{g}(\eta) = \left(\left(\frac{(n-2)}{2}\right)^{2} - c_{3}e^{-2t_{k}}\right)\eta^{2} - \frac{n-2}{n}c_{1}w(t_{k})^{\frac{4}{n-2}}\eta^{\frac{2n}{n-2}}.
	\end{equation*}
	First, since $1 \leq \eta \leq \frac{w(t)}{w(t_{k})} \leq \frac{\varepsilon_{0}}{w(t_{k})}$, we have that
	\begin{equation*}
	\frac{w(t_{k})^{\frac{4}{n-2}}\left(\eta^{\frac{2n}{n-2}} - 1\right)}{\eta^{2} - 1} \leq cw(t_{k})^{\frac{4}{n-2}}\eta^{\frac{4}{n-2}} \leq c\varepsilon_{0}^{\frac{4}{n-2}},
	\end{equation*}
	and we observe that 
	\begin{equation*}
	w(t_{k})^{\frac{4}{n-2}}\int_{1}^{\frac{w(t)}{w(t_{i})}} \frac{\eta^{\frac{4}{n-2}}}{\sqrt{\eta^{2}-1}}d\eta \leq c. 
	\end{equation*}
	
	Now
	\begin{equation*}
	\left(\frac{\eta^{2}-1}{\overline{g}(\eta) - \overline{g}(1)}\right)^{\frac{1}{2}} \leq \frac{2}{n-2} + ce^{-2t_{k}} + c	\frac{w(t_{k})^{\frac{4}{n-2}}\left(\eta^{\frac{2n}{n-2}} - 1\right)}{\eta^{2} - 1}.
	\end{equation*}
	
	Finally, since 
	\begin{equation*}
	\int_{1}^{\frac{w(t)}{w(t_{k})}} \frac{d\eta}{\sqrt{\eta^{2}-1}} \leq c + \ln\frac{w(t)}{w(t_{k})},
	\end{equation*}
	we obtain
	\begin{equation*}
	\int_{1}^{\frac{w(t)}{w(t_{k})}} \frac{d\eta}{\sqrt{\bar{g}(\eta) - \bar{g}(1)}} \leq \left(\frac{2}{n-2} + ce^{-2t_{k}}\right)\ln\frac{w(t)}{w(t_{k})} + c.
	\end{equation*}
	
	From inequality $\eqref{25}$, we get
	\begin{equation}\label{26}
	t - t_{k} \leq \left(\frac{2}{n-2} + ce^{-2t_{k}}\right)\ln\frac{w(t)}{w(t_{k})} + c
	\end{equation}
	for all $t\in(t_{k}, t_{k}^{*})$.
	
	In order to estimate $t-t_{k}$ from below, we first observe that the second inequality in $\eqref{24}$ implies that
	\begin{equation*}
	w_{tt} - \left(\left(\frac{n-2}{2}\right)^{2} + ce^{-2t_{k}}\right)w\leq 0.
	\end{equation*}
	Then the function $w_{t}^{2} - \left(\left(\frac{n-2}{2}\right)^{2} + ce^{-2t_{k}}\right)w^{2}$ is decreasing in $(t_{k},t_{k}^{*})$, and therefore
	\begin{equation*}
	w_{t}^{2} - \left(\left(\frac{n-2}{2}\right)^{2} + ce^{-2t_{k}}\right)w^{2} \leq - \left(\left(\frac{n-2}{2}\right)^{2} + ce^{-2t_{k}}\right)w^{2}(t_{k}).
	\end{equation*}
	
	Hence
	\begin{equation*}
	w_t \leq \sqrt{\left(\left(\frac{n-2}{2}\right)^{2} + ce^{-2t_{k}}\right)(w^{2} - w^{2}(t_{k}))},
	\end{equation*}
	and then
	\begin{equation*}
	t-t_{k} = \int_{w(t_{k})}^{w(t)}\frac{dt}{dw} dw\geq \left(\frac{2}{n-2} - ce^{-2t_{k}}\right)\int_{w(t_{k})}^{w(t)} \frac{dw}{\sqrt{w^{2} - w^{2}(t_{k})}}.
	\end{equation*}
	
	Together with inequality $\eqref{26}$, we get for $t_{k} \leq t \leq t_{k}^{*}$, that
	\begin{equation}\label{27}
	\left(\frac{2}{n-2} - ce^{-2t_{k}}\right)\ln\frac{w(t)}{w(t_{k})} \leq t - t_{k} \leq 	\left(\frac{2}{n-2} + ce^{-2t_{k}}\right)\ln\frac{w(t)}{w(t_{k})} + c.
	\end{equation}
	Similarly one can prove that, for $\overline{t}_{k} \leq t \leq t_{k}$, it holds
	\begin{equation}\label{28}
	\left(\frac{2}{n-2} - ce^{-2\overline{t}_{k}}\right)\ln\frac{w(t)}{w(t_{k})} \leq t_{k} - t \leq 	\left(\frac{2}{n-2} + ce^{-2\overline{t}_{k}}\right)\ln\frac{w(t)}{w(t_{k})} + c.
	\end{equation}



	Once we get the above inequalities, let us go back to the estimates of the terms $I_1$ and $I_2$ in inequality $\eqref{30}$. Recall that by \eqref{23}, $|A(\mathcal{U})| \leq c|x|^{2-n}$, and therefore
	\begin{equation*}
	I_{2} \leq c(r_{k}^{*})^{2} = ce^{-2t_{k}^{*}}.
	\end{equation*}
	
	From the first inequality in $\eqref{27}$, we obtain
	\begin{equation*}
	w(t) \leq w(t_{k})\exp\left(\left(\frac{n-2}{2} + ce^{-2t_{k}}\right)(t-t_{k})\right),
	\end{equation*}
	which implies
	\begin{equation}\label{31}
	v(x) \leq cw(t_{k})\exp\left( - \left(\frac{n-2}{2} +ce^{-2t_{k}}\right)t_{k}\right)r^{2-n-ce^{-2t_{k}}}.
	\end{equation}
	Recall that, by the spherical Harnack inequality, Corollary \ref{DesHarnack}, for each coordinate function, $u_{i} \leq Cr^{\frac{2-n}{2}}$, $|\nabla u_{i}| \leq Cr^{-1}u_{i}$, and $|\nabla^{2}u_{i}|\leq Cr^{-2}u_{i}$, so
	\begin{equation*}
	|A(\mathcal{U})| \leq Cr^{\frac{2-n}{2}}u.
	\end{equation*}
	
	Using the estimate $\eqref{31}$ we obtain
	\begin{equation*}
	\displaystyle	I_{1} \leq cw(t_{k})e^{-\frac{n-2}{2}t_{k} - ce^{-2t_{k}}t_{k}}\int_{B_{r_{k}}\backslash B_{r_{k}^{*}}}|x|^{3-\frac{3n}{2} - ce^{-2t_{k}}}dx,
	\end{equation*}
	and so
	\begin{equation*}
	I_{1} \leq c w(t_{k}) e^{-2t_{k}}.
	\end{equation*}
	
	Therefore, from $\eqref{30}$ and the estimates for $I_1$ and $I_2$, we get 
	\begin{equation*}
	w^{2}(t_{k}) \leq cw(t_{k})e^{-2t_{k}} + ce^{-2t_{k}^{*}}.
	\end{equation*}
	
	Passing to subsequences, if necessary, we can suppose either
	\begin{equation}\label{32}
	w^{2}(t_{k}) \leq cw(t_{k})e^{-2t_{k}}
	\end{equation}
	or
	\begin{equation}\label{33}
	w^{2}(t_{k}) \leq ce^{-2t_{k}^{*}}.
	\end{equation}
	
	Define $L_{k} = - \frac{2}{n-2}\log w(t_{k})$ and choose $\delta > 0$ small. Then, from the first inequality in $\eqref{28}$, we get
	\begin{equation}\label{34}
	t_{k} - \bar{t}_{k} \geq (1-\delta) L_{k} - c,
	\end{equation}
	and adding to the first inequality in $\eqref{27}$, we obtain
	\begin{equation}\label{35}
	t_{k}^{*} - \bar{t}_{k} \geq (2 - 2\delta)L_{k} - c.
	\end{equation}
	
	If inequality $\eqref{32}$ holds, then $w(t_{k}) \leq ce^{-2t_{k}}$ and so $L_{k}\geq \frac{4}{n-2}t_{k} - c$. From inequality $\eqref{34}$, we conclude
	\begin{equation*}
	t_{k} - t_{k}^{*} \geq (1-\delta)\frac{4}{n-2}t_{k} -c,
	\end{equation*}
	and consequently
	\begin{equation*}
	\bar{t}_{k} \leq \left(\frac{n-6}{n-2} + \frac{4\delta}{n-2}\right)t_{k} + c.
	\end{equation*}
	The inequality above gives us a contradiction since $t_{k}^{*} \geq t_{k} \geq \bar{t}_{k} \rightarrow \infty$ as $k \rightarrow \infty$ and on the other hand  $\frac{n-6}{n-2} + \frac{4\delta}{n-2} < 0$ by our assumption $3 \leq n \leq 5$.
	
	If inequality $\eqref{33}$ holds, then $L_{k}\geq \frac{2}{n-2}t_{k}^{*} + c$. From inequality $\eqref{35}$, we get 
	\begin{equation*}
	\bar{t}_{k} \leq t_{k}^{*} - (2-2\delta)L_{k} + c,
	\end{equation*}
	and so
	\begin{equation*}
	\bar{t}_{k} \leq \left(\frac{n-6}{n-2} + 2\delta\right)t_{k}^{*} + c.
	\end{equation*}
	If $3 \leq n \leq 5$, this is again a contradiction by the same reasons as before. Then we conclude that $\lim_{x\rightarrow 0} u(x)|x|^{\frac{n-2}{2}} = 0$, and the result follows as consequence of Lemma \ref{l2}.
	
\end{proof}

As a consequence of the removable singularity theorem, we can now establish a fundamental lower bound.

\begin{corollary} Assume $3 \leq n \leq 5$ and that the potential $A$ satisfies the hypotheses \ref{H1} and \ref{H2}. Let $\mathcal{U}$ be a positive solution to the system \eqref{eq000} in $B^{n}_{1}(0)\backslash\{0\}$. If 0 is a nonremovable singularity, then there exists $c > 0$ such that
	\begin{equation*}
	|\mathcal{U}|(x) \geq cd_{g}(x,0)^{\frac{2-n}{2}}
	\end{equation*}
	for each $i$ and $0 < d_{g}(x,0) < \frac{1}{2}$.	
\end{corollary}

\begin{proof} Suppose by contradiction that this is not true.
	
	Then $\displaystyle\liminf_{t\rightarrow \infty}w(t) = 0$, where $w(t) = r^{\frac{n-2}{2}}\overline{u}(r)$, $u= u_{1}+u_2$ and $t= - \log r$. As 0 is a nonremovable singularity, we also have $\displaystyle\limsup_{t \rightarrow \infty}w(t) > 0$, otherwise we contradict the Lemma \ref{l2} . Therefore there exists a sequence $t_{k}\rightarrow \infty$ such that $w'(t_{k}) = 0$ and $\displaystyle\lim_{k\rightarrow \infty}w(t_{k}) = 0$. So, if $r_{k} = e^{-t_{k}}$ we can check that 
	$$\begin{array}{rcl}
	\displaystyle P(r_{k},\mathcal{U})& \leq & \displaystyle\sigma_{n-1}\left(\frac{1}{2}\partial_{t}w(t_{k})^{2} - \frac{1}{2}\left(\frac{n-2}{2}\right)^{2}w^{2}(t_{k})\right.\\
	& &\displaystyle\left.+ \frac{(n-2)^2}{8}|W|^{\frac{2n}{n-2}}(t_{k})\right)(1 + o(1))
	\end{array}$$
	where $|W|^{2} = w_{1}^{2} + w_{2}^2$, and $w_{i}(t) = r^{\frac{n-2}{2}}\overline{u}_{i}(r)$. But, in this case
	\begin{equation*}
	P(\mathcal{U}) = \lim_{i\rightarrow \infty}P(r_{k},\mathcal{U}) = 0,
	\end{equation*}
	which is a contradiction. This finishes the proof.
\end{proof}

\section{Convergence to a Radial Solution} \label{crs}

Our main goal in this section is to prove that a local singular solution to our system is asymptotic to a radial Fowler-type solution, near the nonremovable isolated singularity.

\begin{theorem}
	Suppose that $\mathcal{U}$ is a solution of the system \eqref{S} in the punctured ball $B_{1}^{n}(0)\backslash \{0\}$ and that the potential $A$ satisfies the hypotheses \ref{H1} and \ref{H2}. If there exist positive constants $c_{1}$ and $c_{2}$ such that 
	\begin{equation}\label{hip}
	c_{1} |x|^{\frac{2-n}{2}}\leq |\mathcal{U}|(x)\leq c_2 |x|^{\frac{2-n}{2}}
	\end{equation}
	then there exists a Fowler-type solution $\mathcal{U}_{0} = u_0 \Lambda$ of \eqref{LS}, where $u_0$ is a Fowler solution such that 
	\begin{equation}
	\mathcal{U}(x) = (1 + O(|x|^{\alpha}))\mathcal{U}_{0}(x)
	\end{equation}
	as $x \rightarrow 0$, for some $\alpha > 0$.
\end{theorem}
\begin{proof} First we observe that \eqref{hip} implies that the origin is a nonremovable singularity. Thus, by Theorem \ref{lower} we get that $P(\mathcal U)< 0$.

	
	Now, in cylindrical coordinates, consider $v_{i}(t,\theta) = |x|^{\frac{2-n}{2}}u_i(x)$, where $t = -\log |x|$ and $\theta = \frac{x}{|x|}$. 
	
	
	Let $\{\tau_k\}$ be a sequence of real numbers such that $\tau_k\rightarrow \infty$. Consider the translated sequence $v_{i,k}(t,\theta) = v_{i}(t + \tau_k, \theta)$ defined in $(-\tau_k, \infty)\times \mathbb{S}^{n-1}$. By $\eqref{hip}$ we get that
	\begin{equation*}
	c_1 \leq |\mathcal{V}_{k}(t,\theta)| \leq c_2,
	\end{equation*}
	where $\mathcal V_k=(v_{1,k},v_{2,k})$. Consequently, by standard elliptic estimates, we get the uniform boundedness of any derivative for $t > 0$ .  Since $v_{i,k}$ satisfies,
	\begin{equation*}
	\mathcal{L}_{\hat{g_{k}}}(v_{i,k}) - \sum_{j}B_{ij}v_{j,k} + \frac{n(n-2)}{4}|\mathcal{V}_{k}|^{\frac{4}{n-2}}v_{i,k} = 0,
	\end{equation*}
	where $\mathcal L_{\hat{g_{k}}}$ and $B_{ij}$ are given by \eqref{eq02} and $\hat{g}_k:=dt^2+d\theta^2+O(e^{-2t}) \rightarrow  dt^2 +d\theta^2$. Standard elliptic estimates imply that there exists a subsequence, also denoted by $v_{i,k}$, which converges in the $C^2_{loc}$ topology, to a positive solution of
	\begin{equation*}
	\partial_{t}^{2}v_{i,0} + \Delta_{\mathbb{S}^{n-1}}v_{i,0} - \frac{(n-2)^2}{4}v_{i,0} + \frac{n(n-2)}{4}|\mathcal{V}_0|^{\frac{4}{n-2}}v_{i,0} = 0,
	\end{equation*}
	defined in the whole cylinder. By the characterization result given by Theorem \ref{class}, such limit is a Fowler-type solution and we know that there exists a Fowler solution $v_{\varepsilon}$ and a vector in the unit sphere with positive coordinates $\Lambda$ such that $ \mathcal{V}_{\varepsilon}(t) = \Lambda v_{\varepsilon}(t)$. Hence $\mathcal{V}_{\varepsilon}$ does not depend on $\theta$, then we necessarily have that any angular derivative $\partial_{\theta}v_{i,k}$ converges uniformly to zero. 
	
	Besides, we claim that
	$$\begin{array}{rcl}
	   v_{i,k}(t,\theta) & = & \overline{v}_{i,k}(t)(1+o(1))\\
	\nabla v_{i,k}(t,\theta) & = & -\overline{v}'_{i,k}(t)(1+o(1)),
	  \end{array}
$$
	as $t \rightarrow \infty$. In fact, suppose that the first equality above is false. Then there exist $\overline\varepsilon > 0$ and sequences $\tau_k \rightarrow \infty$, $\theta_k \rightarrow \theta \in \mathbb{S}^{n-1}$ such that
	\begin{equation*}
	\left| \frac{v_{i,k}(\tau_k, \theta_k)}{\overline{v}_{i,k}(\tau_k)} -1\right| \geq \overline\varepsilon
	\end{equation*}
	for some $i \geq 1$. This is a contradiction because, after passing to a subsequence, $\mathcal{V}_k$ converges to a rotationally symmetric Fowler-type solution $\mathcal{V}_0$. The second inequality follows from similar arguments.
	
	
	In the cylindrical setting the Pohozaev integral $P(t,\mathcal{V}) = P(e^{-t}, \mathcal{U})$ becomes
	\begin{equation*}
	\begin{aligned}
	P(t,\mathcal{V}) := \int_{t\times\mathbb{S}^{n-1}}\left( \frac{1}{2}|\partial_t \mathcal V|^2 - \frac{1}{2}|\nabla_{\theta}\mathcal V|^2 - \frac{(n-2)^2}{8}|\mathcal V|^{2} + \frac{(n-2)^2}{8} |\mathcal{V}|^{\frac{2n}{n-2}}\right)d\sigma_1.
	\end{aligned}
	\end{equation*}
	Hence
	\begin{equation}\label{eq002}
	P(\mathcal{V}_{\varepsilon}) := P(0,\mathcal{V}_\varepsilon) = \lim_{k\rightarrow \infty}P(0, \mathcal{V}_k) = \lim_{k\rightarrow \infty}P(\tau_k, \mathcal{V}) = P(\mathcal{V}).
	\end{equation}
	
	So we can conclude that the necksize $\varepsilon$ of the limit function is independent of the sequence of numbers $\tau_k$. Therefore, for each sequence $\tau_k\rightarrow\infty$ the correspondent sequence $\mathcal V_k$ converges to a function $\mathcal V_{\varepsilon,T}(t)=\Lambda v_\varepsilon(t+T)$, with $\Lambda\in\mathbb S^{1}_+$, for some $T\in\mathbb R$ which depends on the sequence $\tau_k$.
	
	We will show that there exists $T_0\in\mathbb R$ such that $\mathcal V_k$ converges to $\mathcal V_{\varepsilon,T_0}$ for any sequence $\tau_k\rightarrow\infty$. The ideia is to use a delicate rescaling argument due originally to Leon Simon. In order of that we will prove several claims using the Jacobi fields studied in subsection \ref{jf} as a tool.
	
	Let $T_{\varepsilon}$ be the period of $\mathcal{V}_{\varepsilon}$ and $A_{\tau} = \displaystyle\sup_{t\geq 0}|\partial_{\theta}\mathcal{V}_{\tau}|$, where $\mathcal{V}_{\tau}(t,\theta) = \mathcal{V}(t + \tau,\theta)$. Note that $A_\tau<\infty$, since $|\partial_{\theta}\mathcal{V}_{\tau}|$ converges uniformly to zero as $t\rightarrow\infty$.\\
	
	\noindent\textbf{Claim 1:} For every $c > 0$, there exists a positive integer $N$ such that, for any $\tau > 0$, either
	\begin{enumerate}
		\item $A_{\tau} \leq ce^{-2\tau}$ or
		\item $A_{\tau}$ is attained at some point in $I_{N}\times \mathbb{S}^{n-1}$, where $I_{N} = [0,NT_{\varepsilon}]$.
	\end{enumerate}
	
	Suppose the Claim is not true. 
	
	Then there exist a constant $c>0$ and sequences $\tau_k, s_k \rightarrow \infty, \theta_k \in \mathbb{S}^{n-1}$ such that $|\partial_{\theta}\mathcal{V}_{\tau}|(s_k, \theta_k) = A_{\tau_k}$ and $A_{\tau_k} > c e^{-2\tau_k}$ as $k\rightarrow \infty$. Then we can translate back further $s_k$ and define $\tilde{v}_{i,k}(t,\theta) = v_{i,k}(t+s_k, \theta)$. Define $\varphi_{i,k} = A_{\tau_k}^{-1}\partial_{\theta}\tilde{v}_{i,k}$ and note that $|\boldsymbol{\varphi_{k}}|\leq 1$, where $\boldsymbol{\varphi_{k}} = (\varphi_{1,k}, \varphi_{2,k})$. Now, we have
	\begin{equation*}
	\mathcal{L}_{\hat{g}_k}(\tilde{v}_{i,k}) - \sum_{j}\tilde{B}_{ij}\tilde{v}_{j,k} +  \frac{n(n-2)}{4}|\tilde{\mathcal{V}}_{k}|^{\frac{4}{n-2}}\tilde{v}_{i,k} = 0,
	\end{equation*}
	where the quantities with tilde are the originals replaced $t$ by $t+\tau_k + s_k$. This implies that 
	\begin{equation*}	
	L_{cyl}(\tilde{v}_{i,k}) +  \frac{n(n-2)}{4}|\tilde{\mathcal{V}}_{k}|^{\frac{4}{n-2}}\tilde{v}_{i,k} = L_{cyl}(\tilde{v}_{i,k}) - \mathcal{L}_{\hat{g}_{k}}(\tilde{v}_{i,k}) + \sum_{j}\tilde{B}_{ij}\tilde{v}_{j,k}.
	\end{equation*}
	Taking derivative with respect to $\theta$ and multipling by $A_{\tau_k}^{-1}$, we get
$$L_{cyl}(\varphi_{i,k}) +  \frac{n(n-2)}{4}|\tilde{\mathcal{V}}_{k}|^{\frac{4}{n-2}}\varphi_{i,k}+n|\tilde{\mathcal V}_k|^{\frac{4}{n-2}-2}\tilde v_{ik}\langle \tilde{\mathcal V}_k,\boldsymbol{\varphi_{k}}\rangle $$
$$= L_{cyl}(\varphi_{i,k}) - \mathcal{L}_{\hat{g}_{k}}(\varphi_{i,k}) + \sum_{j}\tilde{B}_{ij}\varphi_{j,k}.$$

	From \eqref{cur} we have
	\begin{equation*}
	\mathcal{L}_{\hat{g}_{k}}(\varphi_{i,k}) = \Delta_{\hat{g}_{k}} \varphi_{i,k} - \frac{(n-2)^2}{4}\varphi_{i,k} -\frac{n-2}{4}e^{-t}\partial_r\log|g_k|\circ\Phi(t,\theta)\varphi_{i,k}.
	\end{equation*}
	But using the fact $\hat{g}_k = dt^2 + d\theta^2 + O(e^{-2t})$ and the local expression of the laplacian in this metric, we find that 
	\begin{equation*}
	\Delta_{\hat{g}_{k}}\varphi_{i,k} = \partial^2_{t}\varphi_{i,k} + \Delta_{\mathbb{S}^{n-1}}\varphi_{i,k} + O(e^{-2(t+\tau_k + s_k)}).
	\end{equation*}
	This implies that 
	\begin{equation*}
	L_{cyl}(\varphi_{i,k}) +  \frac{n(n-2)}{4}|\tilde{\mathcal{V}}_{k}|^{\frac{4}{n-2}}\varphi_{i,k} + n|\tilde{\mathcal{V}}_{k}|^{\frac{4}{n-2}-1}\tilde{v}_{i,k}\langle \tilde{\mathcal{V}}_k ,  \boldsymbol{\varphi_{k}}\rangle  = A_{\tau_k}^{-1}e^{-2(\tau_k + s_k)}O(e^{-2t})
	\end{equation*}
	where $\boldsymbol{\varphi_{k}} = (\varphi_{1,k}, \varphi_{2,k})$. 
	
	Now we can use elliptic theory to extract a subsequence $\varphi_{i,k}$ which converges in compact subsets to a nontrivial and bounded Jacobi field $\boldsymbol{\varphi} = (\varphi_{1}, \varphi_{2})$ which satisfies the following system
	\begin{equation*}
	L_{cyl}(\varphi_{i}) +  \frac{n(n-2)}{4}v_{\varepsilon}^{\frac{4}{n-2}}\varphi_{i} + n\Lambda_{i}\langle \Lambda, \boldsymbol{\varphi}\rangle v_{\varepsilon}^{\frac{4}{n-2}} = 0.
	\end{equation*}
	Since each coordinate function of the limit $\varphi_{i}$ has no zero eingencomponent relative to $\Delta_{\theta}$, we get a contradiction because a Jacobi field with such property is necessarily unbouded. This proves the Claim 1.
	
	
	Now suppose we have a sequence $v_{i,k}(t,\theta)$  converging to  $\Lambda_{i}v_{\varepsilon}(t+T)$ as $k \rightarrow\infty$. Define
	\begin{equation*}
	w_{i,k}(t,\theta) = v_{i,k}(t,\theta) - \Lambda_{i}v_{\varepsilon}(t+T).
	\end{equation*}
	Set $\mathcal{W}_{k}=(w_{1,k},w_{2,k})$, $\eta_{k} = b\displaystyle\max_{I_{N}}|\mathcal{W}_{k}|$, $\overline{\eta}_{k} = \eta_{k}+ e^{-(2-\delta)\tau_k}$ and $\varphi_{i,k} = \overline\eta_{k}^{-1}w_{i,k}$, where $\delta > 0$ is a small number and $b>0$ is a fixed number to be chosen later. Note that $|(\varphi_{1,k},\varphi_{2,k})|\leq b^{-1}$ on $I_N$. Then
	\begin{equation}\label{eq001}
	\mathcal{L}_{\hat{g}_k}(w_{i,k})  +  \frac{n(n-2)}{4}\left(|\mathcal{V}_{k}|^{\frac{4}{n-2}}v_{i,k} - \Lambda_{i}v_{\varepsilon}^{\frac{n+2}{n-2}} \right)=  E_{i,k}
	\end{equation}
	where $E_{i,k} = \sum_{j}B_{ij}v_{j,k} + 	\Lambda_{i}(L_{cyl} - \mathcal{L}_{\hat{g}_k})v_{\varepsilon}$. First note that by \eqref{eq02} we get that $E_{i,k} = O(e^{-2(\tau_k+t)})$ when $t \rightarrow \infty$. Second, observe that 
	\begin{equation*}
	|\mathcal{V}_{k}|^{\frac{4}{n-2}}v_{i,k} - \Lambda_{i}v_{\varepsilon}^{\frac{n+2}{n-2}}  = |\mathcal{V}_{k}|^{\frac{4}{n-2}}w_{i,k} + \Lambda_{i}v_{\varepsilon}\frac{|\mathcal{V}_{k}|^{\frac{4}{n-2}} - v_{\varepsilon}^{\frac{4}{n-2}}}{|\mathcal V_k|^2-v_\varepsilon^2}\sum_jw_{j,k}(v_{j,k}+\Lambda_iv_\varepsilon).
	\end{equation*}
	Multipling \eqref{eq001} by $\overline\eta^{-1}$ and taking the limit $k\rightarrow\infty$ we get
	\begin{equation*}
	L_{cyl}(\varphi_{i}) +  \frac{n(n-2)}{4}v_{0}^{\frac{4}{n-2}}\varphi_{i} + n\Lambda_{i}\langle \Lambda, \boldsymbol{\varphi}\rangle v_{\varepsilon}^{\frac{4}{n-2}} = 0,
	\end{equation*}
	on the whole cylinder, where $\boldsymbol{\varphi}=(\varphi_1,\varphi_2)$ is a Jacobi field. \\
	
	\noindent\textbf{ Claim 2:} The Jacobi field $\boldsymbol{\varphi}$ is bounded for $t \geq 0$.
	
	To prove this claim we will use the analysis done in Section \ref{jf}.
	
	By the spectral decomposition for the Laplacian in the sphere, we know that it is possible to write the Jacobi field as
	\begin{equation*}
	\boldsymbol{\varphi} = a_{1}\phi^{1}_{\varepsilon,0} + a_{2}\phi^{2}_{\varepsilon,0} + a_{3}\phi^{3}_{\varepsilon,0} + a_{4}\phi^{4}_{\varepsilon,0} + \boldsymbol{\tilde{\varphi}}
	\end{equation*}
	where $\phi^{i}_{\varepsilon,0}$ are the linearly independents Jacobi fields corresponding to the  eingencomponent independent of $\theta$, and $\boldsymbol{\tilde{\varphi}}$ denotes the projection onto the orthogonal complement. We also know that the functions $\phi_{\varepsilon,0}^{1}$ and $\phi_{\varepsilon,0}^{3}$ are bounded and $\phi_{\varepsilon,0}^{2}$ and $\phi_{\varepsilon,0}^{4}$ are linearly growing. 
	
	Let us show that $\boldsymbol{\tilde{\varphi}}$ is bounded by proving that each $\partial_{\theta}\tilde{\varphi}_{i} = \partial_{\theta}\varphi_{i}$ is bounded for $t \geq 0$. In fact, the function $\partial_{\theta}\varphi_{i}$ is the limit of $\overline{\eta}_{k}^{-1}\partial_{\theta}v_{i,k}$, and we can suppose that $\partial_{\theta}\varphi_{i}$ is nontrivial, otherwise the result is immediate. 
	
	
	
	If the first item of Claim 1 happens then
	$$\sup_{t\geq 0}\left(\overline{\eta}_{k}^{-1}|\partial_{\theta} v_{i,k}|\right)\leq \frac{ce^{-2\tau_k}}{\eta_k+e^{-(2-\delta)\tau_k}}\leq C.$$
	While if the second item of Claim 1 is true then
	$$\sup_{t\geq 0}\left(\overline{\eta}_{k}^{-1}|\partial_{\theta} v_{i,k}|\right) \leq \sup_{t\geq 0}\left(\overline{\eta}_{k}^{-1}|\partial_{\theta} \mathcal{V}_{k}|\right) = \sup_{t\in I_N}\left(\overline{\eta}_{k}^{-1}|\partial_{\theta} \mathcal{V}_{k}|\right) \leq C,$$
	since the sequence $\overline{\eta}_{k}^{-1}|\partial_{\theta} \mathcal{V}_{k}|$ converges in the $C^2_{loc}$ toplogy. Therefore each $\tilde{\varphi}_{i}$ is bounded for $t \geq 0$, hence exponentially decaying.
	
	To end the proof of the Claim 2 we need to show that $a_{2} = a_{4} = 0$. To see this note that the convergence $\varphi_{i,k}=\overline\eta_k^{-1}w_{i,k}\rightarrow\varphi_i$ implies that
	$$\begin{array}{rcl}
	   \mathcal{V}_{k} & = & \Lambda v_{\varepsilon,T}+\overline\eta_k\boldsymbol\varphi+o(\overline{\eta}_{k})\\
	   & = & \Lambda v_{\varepsilon,T}+ \overline{\eta}_{k}(a_{1}\phi_{\varepsilon,0}^{1} + a_{2}\phi_{\varepsilon,0}^{2} + a_{3}\phi_{\varepsilon,0}^{3} + a_{4}\phi_{\varepsilon,0}^{4} + \boldsymbol{\tilde{\varphi}}) + o(\overline{\eta}_{k}),
	  \end{array}
$$
	where $v_{\varepsilon,T}(t)=v_{\varepsilon}(t+T)$.
	
	On the other hand by \eqref{23}, \eqref{eq002} and the Pohozaev identity, Lemma \ref{lemPI},we have that
	\begin{equation*}
	P(0,\mathcal{V}_{k}) = P(\tau_k,\mathcal{V}) = P(\mathcal{V}) + O(e^{-2\tau_k}) = P(v_{\varepsilon,T}) + O(e^{-2\tau_k}).
	\end{equation*}
	Since $\displaystyle\lim_{k\rightarrow \infty}(\overline{\eta}_{k}^{-1}e^{-2\tau_k}) = 0$, we would have a contradiction in case $a_{2}$ or $a_{4}$ is not zero. Thus each $\boldsymbol{\varphi}$ is bounded for $t \geq 0$. 
	
	Now we will show that there exists some $T$ so that the difference between $\mathcal{V}$ and $\mathcal{V}_{\varepsilon,T}=\Lambda v_{\varepsilon,T}$ goes to zero as $t \rightarrow \infty$. 
	
	Since we do not know the correct translation parameter, define $\mathcal{V}_{\tau}(t,\theta) = \mathcal{V}(t+\tau, \theta)$ and $ \mathcal{W}_{\tau}(t,\theta)=  \mathcal{V}_{\tau}(t,\theta)-\Lambda v_{\varepsilon}(t)$. Let $C_1 > 0$ be a fixed constant and consider the interval $I_{N}$ as before in the Claim 1. Set also $\eta(\tau) = b\displaystyle\max_{I_N}|\mathcal{W}_{\tau}|$ and $\overline{\eta}(\tau) = \eta(\tau) + e^{-(2-\delta)\tau}$, where $b>0$ is a fixed constant to be chosen later. We observe that $\eta(\tau)\rightarrow 0$ as $\tau\rightarrow\infty$. Let us prove the following claim.\\
	
	\noindent\textbf{Claim 3:} If $N$, $b$ and $\tau$ are sufficiently large and $\overline{\eta}$ is sufficiently small, then there exists $s$ with $|s|\leq C_1\overline{\eta}(\tau)$ so that $\overline{\eta}(\tau + NT_{\varepsilon} + s) \leq \frac{1}{2}\overline{\eta}(\tau)$.\\
	
	Suppose the claim is not true. Then there exists some sequence $\tau_{k}\rightarrow \infty$ such that $\overline{\eta}(\tau_{k})  \rightarrow 0$ and for any $s$ satisfying $|s|\leq C_1 \overline{\eta}(\tau_k)$ we have that $\overline{\eta}(\tau_{k} + NT_{\varepsilon} + s) > \frac{1}{2}\overline{\eta}(\tau_k)$.	Define $\varphi_{i. k} = \overline{\eta}(\tau_k)^{-1}w_{i, \tau_{k}}$, similarly to the previously claim. We can suppose that $\varphi_{i,k}$ converges in $C^{\infty}$ on compact sets to a Jacobi field, which by Claim 2 it is bounded for $t \geq 0$. So we can write 
	\begin{equation}\label{eq003}
		\boldsymbol{\varphi} = a_{1}\phi^{1}_{\varepsilon,0} + a_{3}\phi^{3}_{\varepsilon,0} + \boldsymbol{\tilde{\varphi}}
	\end{equation}
	where $\boldsymbol{\tilde{\varphi}}$ has exponential decay. Note that $|\boldsymbol\varphi|\leq b^{-1}$ on $I_{N}$, which implies that $a_1$ and $a_3$ are uniformily bounded, independently of the sequence $\tau_k$. Moreover, since $\phi_{\varepsilon,0}^1=v_\varepsilon'\Lambda$ and $\phi_{\varepsilon,0}^3=v_\varepsilon\overline\Lambda$ then 
\begin{equation}\label{eq008}
|a_3v_\varepsilon|\leq |\langle \boldsymbol\varphi,\overline\Lambda\rangle|+|\langle \boldsymbol{\tilde{\varphi}},\overline\Lambda\rangle|\leq b^{-1}+|\boldsymbol{\tilde{\varphi}}|
\end{equation}
	on $I_N$. We know that $v_\varepsilon\geq\varepsilon$ and $\boldsymbol{\tilde\varphi}$ decreases exponentially with a fixed rate, and so we can choose $b$ and $N$ sufficiently large such that $|a_3|$ is sufficiently small.
	
Set $s_{k} = - \overline{\eta}(\tau_k)a_1$ whose absolute value is less than $C_1\overline{\eta}(\tau_k)$ if we choose $C_1$ sufficiently large. Hence for $t \in [0, 2NT_{\varepsilon}]$ we have
	\begin{eqnarray*}
		\mathcal{W}_{\tau_{k} + s_{k}}(t,\theta) &=& \mathcal{V}(t + \tau_k - \overline{\eta}(\tau_{k})a_1, \theta) - \Lambda v_{\varepsilon}(t)\\
		& = & \mathcal{V}_{\tau_{k}}(t - \overline{\eta}(\tau_{k})a_1, \theta) - \Lambda v_{\varepsilon}(t - \overline{\eta}(\tau_{k})a_1)\\
		& & -\overline\eta(\tau_k)a_1  \Lambda\frac{v_{\varepsilon}(t - \overline{\eta}(\tau_{k})a_1) -  v_{\varepsilon}(t) }{-\overline\eta(\tau_k)a_1}\\ 
		& = & \overline{\eta}(\tau_{k})\boldsymbol{\varphi}_{k}(t- \overline{\eta}(\tau_{k})a_1,\theta) - \overline{\eta}(\tau_{k})a_{1}\phi_{\varepsilon,0}^1+ o(\overline{\eta}(\tau_{k}))\\
		&=& \mathcal{W}_{\tau_{k}}(t,\theta) - \overline{\eta}(\tau_{k})a_{1}\phi_{\varepsilon,0}^1 + o(\overline{\eta}(\tau_{k})),
	\end{eqnarray*}
	where $\boldsymbol\varphi_k=(\varphi_{1,k},\varphi_{2,k})$. Here we used the equality $\mathcal W_{\tau_k}=\overline\eta(\tau_k)\boldsymbol\varphi_k+o(\overline\eta(\tau_k))$ and $\boldsymbol{\varphi}_{k}(t- \overline{\eta}(\tau_{k})a_1,\theta)-\boldsymbol{\varphi}_{k}(t,\theta)$ goes to zero as $\tau_k\rightarrow\infty$.
	
Consequently, by \eqref{eq003}, for $t \in [0, 2NT_{\varepsilon}]$ we get that
\begin{equation*}
\mathcal{W}_{\tau_{k} + s_{k}}= \overline{\eta}(\tau_{k})\boldsymbol{\tilde{\varphi}} +\overline\eta(\tau_k)a_3\phi_{\varepsilon,0}^3 + o(\overline{\eta}(\tau_{k})),
\end{equation*}
which implies that
\begin{equation*}
	\max_{I_N}|\mathcal{W}_{\tau_{k} + s_{k} + NT_{\varepsilon}}| = \max_{ [NT_{\varepsilon},2NT_{\varepsilon}]} |\mathcal{W}_{\tau_{k}+s_{k}}|\leq \overline{\eta}(\tau_{k}) \max_{ [NT_{\varepsilon},2NT_{\varepsilon}]}\left(|\boldsymbol{\tilde{\varphi}}|+|a_3v_\varepsilon|\right) + o(\overline{\eta}(\tau_{k})).
\end{equation*}
	
Since $\boldsymbol{\tilde{\varphi}}$ decreases exponentially with a fixed rate, by \eqref{eq008} we can choose $N$ and $b>0$ suficiently large in a way that the last equality implies that 
\begin{equation*}
		\max_{I_N}|\mathcal{W}_{\tau_{k} + s_{k} + NT_{\varepsilon}}| \leq \frac{1}{4}\overline{\eta}(\tau_{k}).
\end{equation*}
On the other hand, note that 
\begin{equation*}
	e^{-(2-\delta)(\tau_{k} + s_{k} +NT_{\varepsilon})} \leq e^{-(2-\delta)NT_{\varepsilon}}\overline{\eta}(\tau_{k})\leq\frac{1}{4}\overline{\eta}(\tau_{k})
\end{equation*}
which implies that $\overline{\eta}(\tau + NT_{\varepsilon} + s) \leq \frac{1}{2}\overline{\eta}(\tau)$, a contradiction. This ends the proof of the Claim 3.

	Once the above claim is proved, using an iterative argument, we are ready to prove that there exists $\sigma$ such that $w_{i,\sigma} \rightarrow 0$ as $t \rightarrow \infty$ for each coordinate. First choose $\tau_0$ and $N$ sufficiently large satisfying the Claim 3 and such that $C_1\overline{\eta}(\tau_0) \leq \frac{1}{2}NT_\varepsilon$. Let $s_0=-\overline\eta(\tau_0)a_1$ be chosen as above. Thus we have $|s_0|\leq C_1\overline\eta(\tau_0)\leq\frac{1}{2}NT_\varepsilon$. Define inductively three sequences as 
	\begin{equation*}
	\begin{array}{l}
	\sigma_k = \tau_0 +\displaystyle \sum_{i=0}^{k-1}s_i\\
	\tau_k = \tau_{k-1} + s_{k-1} + NT_\varepsilon=\sigma_k+kNT_\varepsilon\\
	s_k=-\overline\eta(\tau_k)a_1.
	\end{array}
	\end{equation*}
	By the Claim 3 we get by induction $\overline{\eta}(\tau_k) \leq 2^{-k}\overline{\eta}(\tau_0)$ and $|s_k| \leq 2^{-k-1}NT_\varepsilon$. Hence there exists the limit $\sigma = \lim \sigma_k \leq \tau_0 + NT_\varepsilon$ and then $\tau_k\rightarrow\infty$ as $k\rightarrow\infty$. 
	
	We claim $\sigma$ is the correct translation parameter.
	
	In fact, choose $k$ such that $t =kNT_\varepsilon +[t]$ with $[t] \in I_N$, and write
	\begin{eqnarray*}
		w_{i,\sigma}(t,\theta) &=& v_{i}(t+\sigma, \theta) - \Lambda_{i}v_{\varepsilon}(t)\\
		&=& v_{i}(t + \sigma,\theta) - v_{i}(t + \sigma_k,\theta) + v_{i}(t + \sigma_k,\theta) - \Lambda_{i}v_{\varepsilon}(t).
	\end{eqnarray*}
Since $\partial_tv_i$ is uniform boundedness, then
	$$v_{i}(t + \sigma,\theta) - v_{i}(t + \sigma_k,\theta) =\partial_tv_i(t_0)\sum_{j=k}^\infty s_j=O(2^{-k}),$$
	for some $t_0$. Besides,
	$$v_{i}(t + \sigma_k,\theta) - \Lambda_{i}v_{\varepsilon}(t)=v_i(\tau_k+[t],\theta)- \Lambda_{i}v_{\varepsilon}([t])=w_{i,\tau_k}([t],\theta).$$
	Thus,
	$$\mathcal W_{\sigma}(t,\theta)=\mathcal W_{\tau_k}([t],\theta)+O(2^{k}).$$
	
	Since $b\displaystyle\max_{I_N}|\mathcal W_{\tau_k}|=\eta(\tau_k) \leq \overline{\eta}(\tau_k) \leq 2^{-k}\overline{\eta}(\tau_0)$, we will have $|w_{i, \sigma}(t,\theta)| = O(2^{-k})$ or equivalently, using that $t =kNT_\varepsilon +[t]$, we have
	\begin{equation*}
	|w_{i, \sigma}(t,\theta)| \leq C_1e^{ - \frac{\log 2}{NT_\varepsilon}t}
	\end{equation*}
	which finishes the proof of the theorem.
\end{proof}

As a direct consequency of the results proved in this section we have the following corollary.

\begin{corollary} Suppose that $\mathcal{U}$ is a solution of the system \eqref{S} in the punctured ball $B_{1}^{n}(0)\backslash \{0\}$ and $3\leq n\leq 5$. Then there exists a Fowler-type solution $\mathcal{U}_{0}$ from $\eqref{LS}$ such that 
\begin{equation*}
		\mathcal{U}(x) = (1 + O(|x|^{\alpha}))\mathcal{U}_{0}(x)
\end{equation*}
as $x \rightarrow 0$, for some $\alpha > 0$.
\end{corollary}

%
%


\end{document}